\documentclass[a4paper,12pt,reqno,english]{amsart}
\usepackage[utf8]{inputenc}
\usepackage[T1]{fontenc}
\usepackage{babel}

\usepackage{amssymb}

\usepackage{aliascnt}
\usepackage[bookmarks = false, colorlinks, citecolor = blue, urlcolor = blue]{hyperref}
\usepackage[titletoc, toc, title]{appendix}
\usepackage{tikz}

\numberwithin{equation}{section}
\numberwithin{figure}{section}

\theoremstyle{plain}
  \newtheorem{theorem}{Theorem}

\theoremstyle{plain}
  \newaliascnt{proposition}{theorem}
  \newtheorem{proposition}[proposition]{Proposition}
  \aliascntresetthe{proposition}

\theoremstyle{plain}
  \newaliascnt{lemma}{theorem}
  \newtheorem{lemma}[lemma]{Lemma}
  \aliascntresetthe{lemma}

\theoremstyle{plain}
  \newaliascnt{corollary}{theorem}
  
  \aliascntresetthe{corollary}

\theoremstyle{definition}
  \newtheorem{definition}{Definition}

\theoremstyle{definition}
  \newtheorem{remark}{Remark}

\theoremstyle{definition}
  \newtheorem{notation}{Notation}

\theoremstyle{definition}
  \newtheorem{problem}{Problem}

\theoremstyle{definition}
  \newtheorem{conjecture}{Conjecture}

\addto\extrasenglish{
  
}

\date{}

\begin{document}

\title[Commutation relations for cominuscole parabolics]{Commutation relations for quantum root vectors of cominuscole parabolics}

\author{Marco Matassa}%

\email{marco.matassa@gmail.com, mmatassa@math.uio.no}

\address{Department of Mathematics, University of Oslo, P.B. 1053 Blindern, 0316 Oslo, Norway.}

\begin{abstract}
We prove a result for the commutator of quantum root vectors corresponding to cominuscole parabolics.
Specifically we show that, given two quantum root vectors, belonging respectively to the quantized nilradical and the quantized opposite nilradical, their commutator belongs to the quantized Levi factor. This generalizes the classical result for Lie algebras.
Recall that the quantum root vectors depend on the reduced decomposition of the longest word of the Weyl group.
We show that this result does not hold for all such choices. We conjecture that it holds when the reduced decomposition is appropriately factorized.

\end{abstract}

\maketitle

\section{Introduction}

The aim of this paper is to generalize to the quantum setting the following classical result (the relevant definitions will be recalled in the next section).
Let $\mathfrak{g}$ be a complex simple Lie algebra and $\mathfrak{p}$ a parabolic subalgebra.
Let $\mathfrak{u_+}$, $\mathfrak{u_-}$ and $\mathfrak{l}$ be respectively the nilradical, the opposite nilradical and the Levi factor of $\mathfrak{p}$.
Then we have the commutation relations $[\mathfrak{u}_{+}, \mathfrak{u}_{-}] \subset \mathfrak{l}$.
In particular, this result holds when $\mathfrak{p}$ is of cominuscule type.

Quantized versions of the algebras appearing above can be defined from $U_q(\mathfrak{g})$, the quantized enveloping algebra of $\mathfrak{g}$.
Moreover we can define quantum root vectors using Lusztig's automorphisms. They depend on the choice of the reduced decomposition of the longest word of the Weyl group of $\mathfrak{g}$. The main result of this paper is the following.

\begin{theorem}
\label{thm:main-thm}
Let $\mathfrak{p}$ be a parabolic subalgebra of cominuscule type.
Then there exists a reduced decomposition of the longest word of the Weyl group of $\mathfrak{g}$ such that
\[
[E_{\xi_i}, F_{\xi_j}] \in U_q(\mathfrak{l}), \quad
\forall \xi_i, \xi_j \in \Delta(\mathfrak{u_+}).
\]
\end{theorem}

We will also show that this result is not true for all reduced decompositions.
Nevertheless, we conjecture that it holds for all reduced decompositions which factor through the longest word of the Weyl group of $\mathfrak{l}$. This is formulated in \autoref{con:conjecture}.

The theorem stated above shows that, in the case of parabolic subalgebras of cominuscole type, a significant part of the classical Lie algebra structure is preserved. This is in line with other results that we now briefly recall, which provide some motivations for this study.

The quotient $\mathfrak{g} / \mathfrak{p}$, where $\mathfrak{p}$ is of cominuscole type, corresponds infinitesimally to an irreducible generalized flag manifold.
The quantized versions of these spaces enjoy many properties which are close to their classical counterparts.
Important results were obtained by Heckenberger and Kolb in \cite{flagcalc1, flagcalc2}: they show that these spaces admit a canonical $q$-analogue of the de Rham complex, with the homogenous components having the same dimensions as in the classical case.
The first-order differential calculi coming from this construction can be implemented by commutators with Dirac operators, in the sense of spectral triples, as shown by Krähmer in \cite{qflag}.
In general it is not known, however, if these operators have compact resolvent, which is an important requirement for a spectral triple.
This difficulty motivated the subsequent construction of Krähmer and Tucker-Simmons in \cite{qflag2}, which is seen as a first step towards proving a quantum analogue of the Parthasarathy formula.
This classical result expresses the square of the Dirac operator in terms of the Casimir and some constants, which readily allows the computation of its spectrum.

We believe that \autoref{thm:main-thm} provides a contribution towards this goal.
However, this result can not be applied straightforwardly.
Indeed, in this case what we need are commutation relations between the quantum root vectors $E_{\xi_i}$ and $E_{\xi_j}^*$ (see the discussion in \cite{mat15}, where the case of projective spaces is treated in detail). Here $*$ denotes the involution corresponding to the compact real form of $\mathfrak{g}$.
It seems likely that \autoref{thm:main-thm} should have a counterpart in this setting, but of course this remains to be proven.

The structure of the paper is as follows. In \autoref{sec:notation} we fix our notations and conventions.
In \autoref{sec:commutation} we discuss general results on commutation relations and outline our approach to the proof of the theorem.
In \autoref{sec:convex-ord} we obtain some results related to the convex ordering of the radical roots.
In \autoref{sec:posets} we recall the explicit descriptions of the posets of radical roots.
In \autoref{sec:enumeration} we introduce a specific enumeration of the radical roots.
In \autoref{sec:proof} we complete the proof of the theorem.
Finally in \autoref{sec:counter} we show that the result does not hold for all choices of reduced decomposition.

\section{Notations and conventions}
\label{sec:notation}

In this section we fix some notations and briefly review some facts about complex simple Lie algebras, parabolic subalgebras and quantized enveloping algebras.

\subsection{Parabolic subalgebras}

Let $\mathfrak{g}$ be a finite-dimensional complex simple Lie algebra with a fixed Cartan subalgebra $\mathfrak{h}$.
We denote by $\Delta(\mathfrak{g})$ the root system, by $\Delta^{+}(\mathfrak{g})$ the positive roots and by $\Pi = \{ \alpha_{1}, \cdots, \alpha_{r} \}$ the simple roots.
Denote by $a_{ij}$ the entries of the Cartan matrix and by $(\cdot, \cdot)$ the usual invariant bilinear form on $\mathfrak{h}^{*}$.
In particular, in the simply-laced case we have $(\alpha_{i}, \alpha_{j}) = a_{ij}$.
Let $S \subset \Pi$ be a subset of the simple roots. Then we set
$$
\Delta(\mathfrak{l}) = \mathrm{span}(S) \cap \Delta(\mathfrak{g}), \quad
\Delta(\mathfrak{u}_{+}) = \Delta^{+}(\mathfrak{g}) \backslash \Delta^{+}(\mathfrak{l}).
$$
In terms of these roots we define
$$
\mathfrak{l} = \mathfrak{h} \oplus \bigoplus_{\alpha \in \Delta(\mathfrak{l})} \mathfrak{g}_{\alpha}, \quad
\mathfrak{u}_{\pm} = \bigoplus_{\alpha \in \Delta(\mathfrak{u}_{+})} \mathfrak{g}_{\pm \alpha}, \quad
\mathfrak{p} = \mathfrak{l} \oplus \mathfrak{u}_{+}.
$$
It follows that $\mathfrak{l}$ and $\mathfrak{u}_{\pm}$ are Lie subalgebras of $\mathfrak{g}$.
We call $\mathfrak{p}$ the \emph{standard parabolic subalgebra} associated to $S$ (and omit $S$ from the notation). The subalgebra $\mathfrak{l}$ is reductive and is called the \emph{Levi factor} of $\mathfrak{p}$, while $\mathfrak{u}_{+}$ is a nilpotent ideal of $\mathfrak{p}$ called the \emph{nilradical}. We refer to the roots of $\Delta(\mathfrak{u}_{+})$ as the \emph{radical roots}.
We have the commutation relations $[\mathfrak{u}_{+}, \mathfrak{u}_{-}] \subset \mathfrak{l}$.

The adjoint action of $\mathfrak{p}$ on $\mathfrak{g}$ descends to an action on $\mathfrak{g} / \mathfrak{p}$. The decomposition $\mathfrak{g} = \mathfrak{u}_{-} \oplus \mathfrak{p}$ gives $\mathfrak{g} / \mathfrak{p} \cong \mathfrak{u}_{-}$ as $\mathfrak{l}$-modules.
We say that $\mathfrak{p}$ is of \emph{cominuscule type} if $\mathfrak{g} / \mathfrak{p}$ is a simple $\mathfrak{p}$-module. The following well-known result readily allows to classify all cominuscule parabolics.

\begin{proposition}
A parabolic subalgebra $\mathfrak{p}$ is cominuscule if and only if it corresponds to $S = \Pi \backslash \{\alpha_t\}$, where the simple root $\alpha_t$ appears with multiplicity $1$ in the highest root of $\mathfrak{g}$.
\end{proposition}

Moreover, it is clear from this result that all radical roots contain $\alpha_t$ with multiplicity $1$. The classification of cominuscole parabolics is reported in \autoref{tab:class-comin}, see for example \cite{schubert}.

\begin{table}[h]

\begin{tabular}{|c|c|c|}
\hline
Root system & Dynkin diagram & Nomenclature
\tabularnewline
\hline
\hline

$A_r$ &

\begin{tikzpicture}[
root/.style = {circle, draw = black, fill = white, thick, inner sep = 0pt, minimum size = 3mm},
rootlabel/.style = {inner sep = 0pt, scale = .8}
]

\draw (0,0) -- (1,0) node {};
\draw [dotted] (1,0) -- (2,0) node {};
\draw [dotted] (2,0) -- (3,0) node {};
\draw (3,0) -- (4,0) node {};

\node at (0, 0) [root] {};
\node at (1, 0) [root] {};
\node at (2, 0) [root, fill = black] {};
\node at (3, 0) [root] {};
\node at (4, 0) [root] {};

\node at (0,-.35) [rootlabel] {$1$};
\node at (1,-.35) [rootlabel] {$2$};
\node at (2,-.35) [rootlabel] {$k$};
\node at (3,-.35) [rootlabel] {$r - 1$};
\node at (4,-.35) [rootlabel] {$r$};

\end{tikzpicture}

&

Grassmannian $Gr(k, r)$

\tabularnewline
\hline

$B_r$ &

\begin{tikzpicture}[
root/.style = {circle, draw = black, fill = white, thick, inner sep = 0pt, minimum size = 3mm},
rootlabel/.style = {inner sep = 0pt, scale = .8}
]

\draw (0,0) -- (1,0) node {};
\draw [dotted] (1,0) -- (2,0) node {};
\draw (2,0) -- (3,0) node {};
\draw (3,.05) -- (4,.05) node {};
\draw (3,-.05) -- (4,-.05) node {};
\draw (3.4,.2) -- (3.6,0) node {};
\draw (3.4,-.2) -- (3.6,0) node {};

\node at (0, 0) [root, fill = black] {};
\node at (1, 0) [root] {};
\node at (2, 0) [root] {};
\node at (3, 0) [root] {};
\node at (4, 0) [root] {};

\node at (0,-.35) [rootlabel] {$1$};
\node at (1,-.35) [rootlabel] {$2$};
\node at (2,-.35) [rootlabel] {$r - 2$};
\node at (3,-.35) [rootlabel] {$r - 1$};
\node at (4,-.35) [rootlabel] {$r$};

\end{tikzpicture}

&

Odd dimensional quadric $\mathbb{Q}^{2r - 1}$

\tabularnewline
\hline

$C_r$ &

\begin{tikzpicture}[
root/.style = {circle, draw = black, fill = white, thick, inner sep = 0pt, minimum size = 3mm},
rootlabel/.style = {inner sep = 0pt, scale = .8}
]

\draw (0,0) -- (1,0) node {};
\draw [dotted] (1,0) -- (2,0) node {};
\draw (2,0) -- (3,0) node {};
\draw (3,.05) -- (4,.05) node {};
\draw (3,-.05) -- (4,-.05) node {};
\draw (3.6,.2) -- (3.4,0) node {};
\draw (3.6,-.2) -- (3.4,0) node {};

\node at (0, 0) [root] {};
\node at (1, 0) [root] {};
\node at (2, 0) [root] {};
\node at (3, 0) [root] {};
\node at (4, 0) [root, fill = black] {};

\node at (0,-.35) [rootlabel] {$1$};
\node at (1,-.35) [rootlabel] {$2$};
\node at (2,-.35) [rootlabel] {$r - 2$};
\node at (3,-.35) [rootlabel] {$r - 1$};
\node at (4,-.35) [rootlabel] {$r$};

\end{tikzpicture}

&

Lagrangian Grassmannian $LG(r, 2r)$

\tabularnewline
\hline

$D_r$ &

\begin{tikzpicture}[
root/.style = {circle, draw = black, fill = white, thick, inner sep = 0pt, minimum size = 3mm},
rootlabel/.style = {inner sep = 0pt, scale = .8}
]

\draw (0,0) -- (1,0) node {};
\draw [dotted] (1,0) -- (2,0) node {};
\draw (2,0) -- (3,.25) node {};
\draw (2,0) -- (3,-.25) node {};

\node at (0, 0) [root, fill = black] {};
\node at (1, 0) [root] {};
\node at (2, 0) [root] {};
\node at (3, .25) [root] {};
\node at (3, -.25) [root] {};

\node at (0,-.35) [rootlabel] {$1$};
\node at (1,-.35) [rootlabel] {$2$};
\node at (2,-.35) [rootlabel] {$r - 2$};
\node at (3.7,.25) [rootlabel] {$r - 1$};
\node at (3.4,-.25) [rootlabel] {$r$};

\end{tikzpicture}

& Even dimensional quadric $\mathbb{Q}^{2r - 2}$ 

\tabularnewline
\hline

$D_r$ &

\begin{tikzpicture}[
root/.style = {circle, draw = black, fill = white, thick, inner sep = 0pt, minimum size = 3mm},
rootlabel/.style = {inner sep = 0pt, scale = .8}
]

\draw (0,0) -- (1,0) node {};
\draw [dotted] (1,0) -- (2,0) node {};
\draw (2,0) -- (3,.25) node {};
\draw (2,0) -- (3,-.25) node {};

\node at (0,0) [root] {};
\node at (1,0) [root] {};
\node at (2,0) [root] {};
\node at (3,.25) [root, fill = black] {};
\node at (3,-.25) [root, fill = black] {};

\node at (0,-.35) [rootlabel] {$1$};
\node at (1,-.35) [rootlabel] {$2$};
\node at (2,-.35) [rootlabel] {$r - 2$};
\node at (3.7,.25) [rootlabel] {$r - 1$};
\node at (3.4,-.25) [rootlabel] {$r$};

\end{tikzpicture}

& Orthogonal Grassmannian $OG(r + 1, 2r + 2)$

\tabularnewline
\hline 

$E_6$ &

\begin{tikzpicture}[
root/.style = {circle, draw = black, fill = white, thick, inner sep = 0pt, minimum size = 3mm},
rootlabel/.style = {inner sep = 0pt, scale = .8}
]

\draw (0,0) -- (1,0) node {};
\draw (1,0) -- (2,0) node {};
\draw (2,0) -- (3,0) node {};
\draw (3,0) -- (4,0) node {};
\draw (2,0) -- (2,.75) node {};

\node at (0, 0) [root, fill = black] {};
\node at (1, 0) [root] {};
\node at (2, 0) [root] {};
\node at (3, 0) [root] {};
\node at (4, 0) [root, fill = black] {};
\node at (2,.75) [root] {};

\node at (0,-.35) [rootlabel] {$1$};
\node at (1,-.35) [rootlabel] {$2$};
\node at (2,-.35) [rootlabel] {$3$};
\node at (3,-.35) [rootlabel] {$4$};
\node at (4,-.35) [rootlabel] {$5$};
\node at (2.35,.75) [rootlabel] {$6$};

\end{tikzpicture}

&

Cayley plane $\mathbb{OP}^2$

\tabularnewline
\hline

$E_7$ &

\begin{tikzpicture}[
root/.style = {circle, draw = black, fill = white, thick, inner sep = 0pt, minimum size = 3mm},
rootlabel/.style = {inner sep = 0pt, scale = .8}
]

\draw (0,0) -- (1,0) node {};
\draw (1,0) -- (2,0) node {};
\draw (2,0) -- (3,0) node {};
\draw (3,0) -- (4,0) node {};
\draw (4,0) -- (5,0) node {};
\draw (2,0) -- (2,.75) node {};

\node at (0, 0) [root] {};
\node at (1, 0) [root] {};
\node at (2, 0) [root] {};
\node at (3, 0) [root] {};
\node at (4, 0) [root] {};
\node at (5, 0) [root, fill = black] {};
\node at (2,.75) [root] {};

\node at (0,-.35) [rootlabel] {$1$};
\node at (1,-.35) [rootlabel] {$2$};
\node at (2,-.35) [rootlabel] {$3$};
\node at (3,-.35) [rootlabel] {$4$};
\node at (4,-.35) [rootlabel] {$5$};
\node at (5,-.35) [rootlabel] {$6$};
\node at (2.35,.75) [rootlabel] {$7$};

\end{tikzpicture}

&

(Unnamed) $G_\omega(\mathbb{O}^3, \mathbb{O}^6)$

\tabularnewline
\hline

\end{tabular}

\medskip

\caption{Classification of cominuscole parabolics. The black node corresponds to the simple root $\alpha_t$. If there is more than one, then they are equivalent choices.}
\label{tab:class-comin}

\end{table}

\subsection{Quantized enveloping algebras}

We briefly review some facts about quantized enveloping algebras. General references for this topic are the books \cite{klsc}, \cite{chpr} and \cite{lus-book}.
With the previous conventions for complex simple Lie algebras, let $d_i = (\alpha_i, \alpha_i) / 2$. Let $q \in \mathbb{C}$ and define $q_i = q^{d_i}$.
The \emph{quantized universal enveloping algebra} $U_q(\mathfrak{g})$ is generated by the elements $E_i$, $F_i$, $K_i$, $K_i^{-1}$, for $1\le i\le r$ and with $r$ the rank of $\mathfrak{g}$, satisfying the relations
\begin{gather*}
K_i K_i^{-1} = K_i^{-1} K_i=1,\ \
K_i K_j = K_j K_i, \\
K_i E_j K_i^{-1} = q_i^{a_{ij}} E_j,\ \
K_i F_j K_i^{-1} = q_i^{-a_{ij}} F_j, \\
E_i F_j - F_j E_i = \delta_{ij} \frac{K_i - K_i^{-1}}{q_i - q_i^{-1}},
\end{gather*}
plus the quantum analogue of the Serre relations.
The Hopf algebra structure is defined by
\begin{gather*}
\Delta(K_i)=K_i\otimes K_i, \quad
\Delta(E_i)=E_i\otimes1+ K_i\otimes E_i, \quad
\Delta(F_i)=F_i\otimes K_i^{-1}+1\otimes F_i, \\
S(K_{i}) = K_{i}^{-1}, \quad
S(E_{i}) = - K_{i}^{-1} E_{i}, \quad
S(F_{i}) = - F_{i} K_{i}, \quad
\varepsilon(K_i)=1, \quad
\varepsilon(E_i)=\varepsilon(F_i)=0.
\end{gather*}
For $q \in \mathbb{R}$, the \emph{compact real form} of $U_q(\mathfrak{g})$ is defined by
\[
K_{i}^{*} = K_{i}, \quad
E_{i}^{*} = K_{i} F_{i}, \quad
F_{i}^{*} = E_{i} K_{i}^{-1}.
\]

Let $\mathfrak{l}$ be the Levi factor corresponding to a parabolic subalgebra of $\mathfrak{g}$ defined by $S \subset \Pi$. Then the quantized enveloping algebra of the Levi factor is defined as
\[
U_q(\mathfrak{l}) = \{ \textrm{subalgebra of $U_q(\mathfrak{g})$ generated by $K_i^{\pm 1}$ and $E_j, F_j$ with $j \in S$} \}.
\]
This definition of the quantized Levi factor appears for example in \cite[Section 4]{quantum-flag}.

\section{Commutation relations}
\label{sec:commutation}

\subsection{Quantum root vectors}

Fix a reduced decomposition $w_0 = s_{i_1} \cdots s_{i_N}$ of the longest word of the Weyl group of $\mathfrak{g}$. Here $s_i$ is the reflection corresponding to $\alpha_i$. It is well known that all the positive roots can be obtained as $\beta_k = s_{i_1} \cdots s_{i_{k - 1}} (\alpha_{i_k})$ for $k = 1, \cdots, N$.

Now let $T_i$ be any version of the Lusztig automorphisms (see \autoref{rem:quant-root}).
The quantum root vectors are then defined by $E_{\beta_k} = T_{i_1} \cdots T_{i_{k - 1}} (E_{i_k})$ for $k = 1, \cdots, N$. They depend on the choice of the reduced decomposition of $w_0$. Similarly the quantum root vectors corresponding to the negative roots are defined by $F_{\beta_k} = T_{i_1} \cdots T_{i_{k - 1}} (F_{i_k})$ for $k = 1, \cdots, N$.

\begin{remark}
\label{rem:quant-root}
Lusztig introduces four different automorphisms denoted by $T^{\prime}_{i, e}$ and $T^{\prime \prime}_{i, e}$ with $e = \pm 1$, see \cite[Chapter 37]{lus-book}.
Let $\omega$ be the automorphism defined by
\[
\omega(E_i) = F_i, \
\omega(F_i) = E_i, \
\omega(K_i) = K_i^{-1}.
\]
Similarly let $\sigma$ be the anti-automorphism defined by
\[
\sigma(E_i) = E_i, \
\sigma(F_i) = F_i, \
\sigma(K_i) = K_i^{-1}.
\]
Then the various automorphisms are related by \cite[37.2.4]{lus-book}
\[
\omega T^{\prime}_{i, e} = T^{\prime \prime}_{i, e} \omega, \quad
\sigma T^{\prime}_{i, e} = T^{\prime \prime}_{i, -e} \sigma.
\]
\end{remark}

We will also need the anti-automorphism $\Omega$ defined by
\[
\Omega(E_i) = F_i, \
\Omega(F_i) = E_i, \
\Omega(K_i) = K^{-1}_i, \
\Omega(q) = q^{-1}.
\]
It satisfies the property $\Omega T_i = T_i \Omega$. Therefore we can also write $F_{\beta_k} = \Omega (E_{\beta_k})$.

We will use the notation $A = (a_1, \cdots, a_N)$ and define
\[
E^A = E^{a_1}_{\beta_1} \cdots E^{a_N}_{\beta_N}, \quad
F^A = \Omega (E^A) = F^{a_N}_{\beta_N} \cdots F^{a_1}_{\beta_1}.
\]
Similarly if $\lambda = c_1 \alpha_1 + \cdots + c_r \alpha_r$ is an element of the root lattice we will write
\[
K_\lambda = K_1^{c_1} \cdots K_r^{c_r}.
\]

\subsection{Commutation relations}

Having defined quantum root vectors, we can now discuss their commutation relations.
We will be interested in the case of quantum root vectors corresponding to radical roots.
Before going into that, we give a brief list of references for general results on commutation relations.
In the case of two quantum root vectors $E_{\beta_i}$ and $E_{\beta_j}$, commutation relations were derived in \cite{comm-rel} (see also the proof in \cite{con-pro}).
A version of this formula appears also in \cite{xi}. In this reference we also find a formula for the commutator $[E_{\beta_i}, F_{\beta_j}]$. This formula will be recalled below and will be our main tool.

\begin{remark}
The commutation relations of \cite{comm-rel} can be used to show that the algebra generated by the quantum root vectors $\{E_{\xi_i}\}$ is quadratic (at least when $w_0$ is factorized).
See \cite[Remark 3.18]{qflag2} and \cite[Section 5]{zwi}. For this reason we will not consider these commutation relations, but focus instead on the case $[E_{\xi_i}, F_{\xi_j}]$.
\end{remark}

In the next proposition we will derive some simple results for commutators in relation to $U_q(\mathfrak{l})$, the quantized enveloping algebra of the Levi factor. At this stage we are not assuming yet that the parabolic subalgebra is of cominuscole type.

\begin{proposition}
For any positive root $\beta_i$ we have $[E_{\beta_i}, F_{\beta_i}] \in U_q(\mathfrak{l})$.
Moreover let $\beta_i$, $\beta_j$ be positive roots. If $[E_{\beta_i}, F_{\beta_j}] \in U_q(\mathfrak{l})$ then we also have $[E_{\beta_j}, F_{\beta_i}] \in U_q(\mathfrak{l})$.
\end{proposition}

\begin{proof}
For the first part, using the defining relations of the algebra we get
\[
[E_{\beta_k}, F_{\beta_k}]
=  T_{i_1} \cdots T_{i_{k - 1}} ([E_k, F_k])
=  T_{i_1} \cdots T_{i_{k - 1}} \big( \frac{K_k - K^{-1}_k}{q_k - q^{-1}_k} \big).
\]
The automorphism $T_w$ maps the element $K_\lambda$ to $K_{w(\lambda)}$ by the action of the Weyl group. Since the space $U_q(\mathfrak{l})$ contains all the generators $K_\lambda$ we obtain the result.

For the second part suppose that $[E_{\beta_i}, F_{\beta_j}] \in U_q(\mathfrak{l})$.
Using the properties of the anti-automorphism $\Omega$ we compute
\[
\Omega([E_{\beta_i}, F_{\beta_j}])
= [\Omega(F_{\beta_j}), \Omega(E_{\beta_i})]
= [E_{\beta_j}, F_{\beta_i}].
\]
Since $\Omega$ maps $U_q(\mathfrak{l})$ into itself we obtain the conclusion.
\end{proof}

Therefore to prove \autoref{thm:main-thm} we only need to consider the case $[E_{\xi_i}, F_{\xi_j}]$ with $i > j$.
To proceed we need to recall the formula of \cite[Lemma 3.2]{xi}.
In this result the quantum root vectors are defined using $T^{\prime}_{i, -1}$ (but for our purposes this is irrelevant).

\begin{theorem}[\cite{xi}]
Let $i > j$. Then we have
\[
[E_{\beta_i}, F_{\beta_j}] = \sum_{A, \lambda, B} \sigma(A, \lambda, B) F^A K_\lambda E^B,
\]
where the sums over $A,\lambda,B$ contain only a finite number of terms. Moreover if $\sigma(A, \lambda, B) \neq 0$ then we have $a_j = \cdots = a_N = 0$ and $b_1 = \cdots = b_i = 0$.
\end{theorem}

More explicitely this can be rewritten as
\begin{equation}
\label{eq:commutator}
[E_{\beta_i}, F_{\beta_j}] = \sum_{\substack{a_1, \cdots, a_{j - 1} \\ b_{i + 1}, \cdots, b_N}} \sum_\lambda \sigma(A, \lambda, B) F^{a_{j - 1}}_{\beta_{j - 1}} \cdots F^{a_1}_{\beta_1} K_\lambda E^{b_{i + 1}}_{\beta_{i + 1}} \cdots E^{b_N}_{\beta_N}.
\end{equation}

Recall that the elements of $U_q(\mathfrak{g})$ are graded by the root lattice. Using this grading, we obtain further conditions for the terms appearing on the RHS. We must have
\[
\beta_i - \beta_j = \sum_{k = i + 1}^{N} b_{k} \beta_k - \sum_{k = 1}^{j - 1} a_k \beta_k .
\]
Let us specialize this to the case of radical roots, which we denote by $\xi$. Then we write
\begin{equation}
\label{eq:identity}
\xi_i - \xi_j = \sum_{\xi > \xi_i} c_\xi \xi - \sum_{\xi < \xi_j} c_\xi \xi + \gamma,
\end{equation}
where $\gamma$ contains the sum over the non-radical roots and we have relabeled the various coefficients.
At this stage the coefficients of the terms appearing in $\gamma$ can have both signs. Later on we will consider particular orders such that $\gamma$ will contain only positive terms.

We now come to the main point. Suppose that if \eqref{eq:identity} holds then we have $c_\xi = 0$ for all $\xi > \xi_i$ and $\xi < \xi_j$. In other words, we only have non-radical roots on the RHS of the equality. Then we have $[E_{\xi_i}, F_{\xi_j}] \in U_q(\mathfrak{l})$.
Indeed the condition $c_\xi = 0$ implies that there are no quantum root vectors $E_{\xi_k}$ or $F_{\xi_k}$ appearing on the RHS of \eqref{eq:commutator}.
Using the definition of the quantized enveloping algebra of the Levi factor $U_q(\mathfrak{l})$, the conclusion is immediate.

\begin{remark}
This statement is independent of the chosen version of the automorphisms $T_i$ defining the quantum root vectors. Indeed, using the notation of \autoref{rem:quant-root}, we see that $U_q(\mathfrak{l})$ is stable under the action of $\omega$ and $\sigma$. This follows immediately from its definition.
\end{remark}

Therefore we are led to study the following problem.

\begin{problem}
\label{prob:equality}
Let $i > j$ and consider two radical roots $\xi_i$ and $\xi_j$. Suppose that
\[
\xi_i - \xi_j = \sum_{\xi > \xi_i} c_\xi \xi - \sum_{\xi < \xi_j} c_\xi \xi + \gamma,
\]
where $\gamma$ contains the sum over the non-radical roots (notice that this depends on the choice of the order $<$).
Does this imply that $c_\xi = 0$ for all $\xi > \xi_i$ and $\xi < \xi_j$?
\end{problem}

Clearly the answer to this problem can depend on the choice of the order.
Observe that, if \autoref{prob:equality} has a positive answer for all $i > j$, then it provides a proof of \autoref{thm:main-thm}, as pointed out above. For this reason we will now focus on this problem.

\section{Convex and partial orders}
\label{sec:convex-ord}

\subsection{Convex orders}

The choice of a reduced expression for the longest word of the Weyl group of $\mathfrak{g}$ gives a total order on the positive roots of $\mathfrak{g}$. This is simply given by $\beta_i < \beta_j$ if $i < j$.
Not all total orders can be realized in this way. It is known that they are characterized by the following condition \cite{papi}: if $\alpha, \ \beta$ are positive roots with $\alpha < \beta$ and $\alpha + \beta$ is a root, then we have $\alpha < \alpha + \beta < \beta$.
Such an order is usually called a \emph{convex order}.
Therefore choosing a reduced decomposition is the same as choosing such an order.

As mentioned in the introduction, the main result that we want to prove does not hold for any reduced decomposition. For this reason we will focus on reduced decompositions of a certain type, namely those which can be factorized in terms of the longest word of the Weyl group of $\mathfrak{l}$, the Levi factor. Therefore we will consider the case in which $\xi < \alpha$, for all radical roots $\xi$ and all non-radical roots $\alpha$.
The case $\xi > \alpha$ can be straightforwardly derived by reversing the ordering. For this reason we will focus on the $\xi < \alpha$ case.

Recall that on the positive roots there is a naturally defined \emph{partial order}: we have $\alpha \succ \beta$ if $\alpha - \beta$ is a non-negative linear combination of simple roots.
In the next lemma we will show that any convex order as above will be compatible with the partial order.

\begin{lemma}
Let $<$ be an ordering of the positive roots. Suppose that:
\begin{itemize}
\item the order $<$ is convex on the non-radical roots,
\item $\xi < \alpha$ for any radical root $\xi$ and non-radical root $\alpha$.
\end{itemize}
Then $<$ is a convex order if and only if $\xi \prec \xi^\prime$ implies $\xi < \xi^\prime$.
\end{lemma}

\begin{proof}
($\Rightarrow$)
It will be sufficient to consider the case when $\xi^\prime - \xi = \alpha_k$ for some simple root $\alpha_k$. Indeed the general case will follow by iterating this elementary one, since any positive root can be obtained from a positive root of lower height by adding a simple root.
Now suppose that $<$ is a convex order. Using our assumption that $\xi < \alpha$ and the convex ordering condition we have $\xi < \xi + \alpha_k < \alpha_k$. Since $\xi^\prime = \xi + \alpha_k$ it follows that $\xi < \xi^\prime$.

($\Leftarrow$)
We need to check the convex ordering conditions. Since any positive root is either non-radical or radical we have three cases: $\alpha + \alpha^\prime$, $\xi + \xi^\prime$ and $\xi + \alpha$, where $\alpha, \ \alpha^\prime$ are non-radical roots and $\xi, \ \xi^\prime$ are radical roots.
Convex ordering for the case $\alpha + \alpha^\prime$ holds by assumption. For the case $\xi + \xi^\prime$ we observe that this is never a root: indeed all radical roots contain $\alpha_t$ with multiplicity $1$, including the highest root, while the sum would have multiplicity $2$.
Finally consider the case $\xi + \alpha$ and suppose that it is a root. Then it is a radical root and we have $\xi \prec \xi + \alpha$. Then our assumptions guarantee that $\xi < \xi + \alpha < \alpha$ holds.
\end{proof}

The upshot is that we are free to choose any order on the radical roots, as long as it is compatible with the partial order, in the sense that $\xi \prec \xi^\prime$ implies $\xi < \xi^\prime$.
From this point on we will consider such a convex order, unless otherwise specified.

\subsection{Some lemmata}

We now return to \autoref{prob:equality}. Therefore let $i > j$ and suppose that for some radical roots $\xi_i$ and $\xi_j$ we have an equality of the form
\[
\xi_{i} - \xi_{j} = \sum_{\xi > \xi_{i}} c_{\xi} \xi - \sum_{\xi < \xi_{j}} c_{\xi} \xi + \gamma.
\]
Observe that, with the order as above, the term $\gamma$ is a positive linear combination of simple roots. This is because $\xi < \alpha$, so that non-radical roots can appear only with positive sign.

In the next lemma we will show that many terms in the sums, which are in principle allowed by the chosen order, can not appear if the equality holds.

\begin{lemma}
\label{lem:eli-ord}
If $\xi^{\prime} \succ \xi_{i}$ then $c_{\xi^{\prime}} = 0$.
Similarly if $\xi^{\prime} \prec \xi_{j}$
then $c_{\xi^{\prime}} = 0$.
\end{lemma}

\begin{proof}
The condition $\xi^{\prime} \succ \xi_{i}$ implies $\xi^{\prime} > \xi_{i}$. Write the relevant equation in the form
\[
(1 - c_{\xi^{\prime}}) \xi^{\prime} = \xi^{\prime} - \xi_{i} + \xi_{j} + \sum_{\xi > \xi_{i}, \xi \neq \xi^{\prime}} c_{\xi} \xi - \sum_{\xi < \xi_{j}} c_{\xi} \xi + \gamma.
\]
It is possible to find an element of the Weyl group $w$ such that $w(\xi_{k}) < 0$ for $k < j$ and $w(\xi_{k}) > 0$ for $k \geq j$, see \cite[Section 1.7]{hum}. Since $\xi^{\prime} \succ \xi_{i}$ we have that $\xi^{\prime} - \xi_{i}$ is a non-negative combination of simple roots. If we apply $w$ we find that the RHS is positive. On the other hand the LHS is positive only for $c_{\xi^{\prime}} = 0$, which gives the conclusion.

The second case is similar. The condition $\xi^{\prime} \prec \xi_{j}$ implies $\xi^{\prime} < \xi_{j}$. Then write
\[
(-1+c_{\xi^{\prime}})\xi^{\prime}=-\xi_{i}+\xi_{j}-\xi^{\prime}-\sum_{\xi<\xi_{j},\xi\neq\xi^{\prime}}c_{\xi}\xi+\sum_{\xi>\xi_{i}}c_{\xi}\xi+\gamma.
\]
We can find $w$ such that $w(\xi_{k}) < 0$ for $k \leq i$ and $w(\xi_{k}) > 0$ for $k > i$. Since $\xi_{j} \succ \xi^{\prime}$ we have that $\xi_{j} - \xi^{\prime}$ is a non-negative combination of simple roots. If we apply $w$ we
find that the RHS is positive. On the other hand the LHS is positive only for $c_{\xi^{\prime}} = 0$.
\end{proof}

As a consequence, we only need to consider terms which are not comparable to $\xi_i$ and $\xi_j$ (with respect to the partial order). We define the two sets
\[
\mathcal{B}_i = \{ \xi > \xi_i : \xi \not \succ \xi_i \}, \quad
\mathcal{S}_j = \{ \xi < \xi_j : \xi \not \prec \xi_j \}.
\]
Then we can rewrite our original equality in the form
\[
\xi_{i} - \xi_{j} = \sum_{\xi \in \mathcal{B}_i} c_{\xi} \xi - \sum_{\xi \in \mathcal{S}_j} c_{\xi} \xi + \gamma.
\]

In the next lemma we show that, provided that we can find two positive roots $\xi_S$ and $\xi_B$ satisfying certain properties, we have a positive answer for \autoref{prob:equality}.

\begin{lemma}
\label{lem:eli-diff}
Let $i > j$. Suppose there exist two positive roots $\xi_S$ and $\xi_B$ such that $\xi_S \preceq \xi$ for all $\xi \in \mathcal{B}_{i}$
and $\xi_B \succeq \xi$ for all $\xi \in \mathcal{S}_{j}$.
Consider the differences
\[
\xi_{i} - \xi_{j}
= \sum_{k} a_{k} \alpha_{k}, \quad \xi_S - \xi_B = \sum_{k} b_{k} \alpha_{k}, \quad a_{k}, b_{k} \in \mathbb{Z}.
\]
Suppose there exists some index $m$ such that $a_{m} < b_{m}$. Then it follows that $c_{\xi} = 0$ for all $\xi \in \mathcal{B}_{i}$ and $\xi \in \mathcal{S}_{j}$. In this case \autoref{prob:equality} has a positive answer.
\end{lemma}

\begin{proof}
Denote by $\alpha_t$ the simple root which defines the cominuscule parabolic. Recall that all radical roots contain $\alpha_t$ with multiplicity $1$, while non-radical roots do not contain it.
Then $\xi_i - \xi_j$ does not contain $\alpha_t$, which implies the equality
\[
C = \sum_{\xi \in \mathcal{B}_{i}} c_{\xi}
= \sum_{\xi \in \mathcal{S}_{j}} c_{\xi}.
\]
Using this fact we rewrite the identity in the form
\[
\xi_{i} - \xi_{j}
= C(\xi_S - \xi_B) + \sum_{\xi \in \mathcal{B}_{i}} c_{\xi} (\xi - \xi_S) + \sum_{\xi \in \mathcal{S}_{j}} c_{\xi} (\xi_B - \xi) + \gamma.
\]
Since $\xi_S \preceq \xi$ for any $\xi \in \mathcal{B}_{i}$ and $\xi_B \succeq \xi$ for any $\xi \in \mathcal{S}_{j}$, the terms $\xi - \xi_S$ and $\xi_B - \xi$ are non-negative.
Also recall that the term $\gamma$ is a non-negative linear combination of simple roots.
Now observe that on the LHS the root $\alpha_{m}$ appears with coefficient $a_{m}$, while on the RHS it appears with coefficient greater or
equal than $b_{m}$, provided that $C > 0$. Since by assumption we have $a_{m} < b_{m}$ we cannot have equality for any $C > 0$.
\end{proof}

To make use of this lemma we will have to study the various posets of radical roots.

\section{Posets of radical roots}
\label{sec:posets}

In this section we describe in detail the posets of radical roots corresponding to cominuscule parabolics.
The shapes of the posets can be found for example in \cite{schubert}.
For us it will also be important to know which simple roots connect the various radical roots.

\subsection{$A_r$ series}

The positive roots of $A_r$ are given by $\alpha_{i, j} = \sum_{k = i}^j \alpha_k$ with $1 \leq i \leq j \leq r$.
The poset of the positive roots is completely described by the relations
\[
\alpha_{i, j} - \alpha_{i + 1, j} = \alpha_i, \quad
\alpha_{i, j} - \alpha_{i, j - 1} = \alpha_j.
\]
The cominuscule parabolics are obtained by removing any node $\alpha_t$ with $1 \leq t \leq r$. An illustration of one of these posets is provided in \autoref{fig:poset-a}.

\begin{figure}[h]

\begin{center}
\begin{tikzpicture}[
scale = 0.9,
root/.style = {circle, draw=black, fill=blue!5, thick, inner sep=0pt, minimum size=7mm}
]

\node at (1, 1) [root] (a2) {$\alpha_{2,2}$};
\node at (2, 2) [root] (a3) {$\alpha_{2,3}$};
\node at (3, 3) [root] (a4) {$\alpha_{2,4}$};
\node at (4, 4) [root] (a5) {$\alpha_{2,5}$};

\node at (2, 0) [root] (a6) {$\alpha_{1,2}$};
\node at (3, 1) [root] (a7) {$\alpha_{1,3}$};
\node at (4, 2) [root] (a8) {$\alpha_{1,4}$};
\node at (5, 3) [root] (a9) {$\alpha_{1,5}$};

\draw (a2) -- (a3) node [red, midway, below, sloped] {$\alpha_{3}$};
\draw (a3) -- (a4) node [red, midway, below, sloped] {$\alpha_{4}$};
\draw (a4) -- (a5) node [red, midway, below, sloped] {$\alpha_{5}$};

\draw (a6) -- (a7) node [red, midway, below, sloped] {$\alpha_{3}$};
\draw (a7) -- (a8) node [red, midway, below, sloped] {$\alpha_{4}$};
\draw (a8) -- (a9) node [red, midway, below, sloped] {$\alpha_{5}$};

\draw (a2) -- (a6) node [red, midway, below, sloped] {$\alpha_{1}$};
\draw (a3) -- (a7) node [red, midway, below, sloped] {$\alpha_{1}$};
\draw (a4) -- (a8) node [red, midway, below, sloped] {$\alpha_{1}$};
\draw (a5) -- (a9) node [red, midway, below, sloped] {$\alpha_{1}$};

\end{tikzpicture}
\end{center}

\caption{Poset for the $A_r$ series with $r = 5$ and $t = 2$.}
\label{fig:poset-a}

\end{figure}
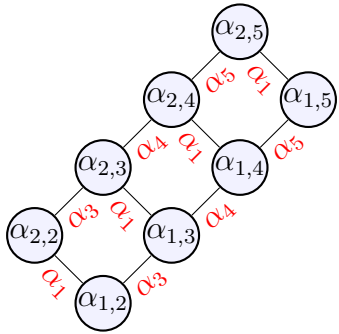

\subsection{$B_r$ series}

The positive roots of $B_r$ are given by
\[
\begin{aligned}
\alpha_{i, j} & = \sum_{k = i}^{j} \alpha_{k}, & 1 \leq i \leq j \leq r, \\
\alpha_{i, j}^{\prime} & = \sum_{k = i}^{j - 1}\alpha_{k} + \sum_{k = j}^{r} 2 \alpha_{k}, & 1 \leq i < j \leq r.
\end{aligned}
\]
The radical roots contain $\alpha_{1}$ with multiplicity $1$. These are
\[
\{\alpha_{1,1},\cdots,\alpha_{1,r}\},\ \{\alpha_{1,r}^{\prime},\cdots,\alpha_{1,2}^{\prime}\}.
\]
The poset is completely described by the following relations
\[
\alpha_{1,j}-\alpha_{1,j-1}=\alpha_{j}, \quad
\alpha_{1,j}^{\prime}-\alpha_{1,j+1}^{\prime}=\alpha_{j}, \quad
\alpha_{1,r}^{\prime}-\alpha_{1,r}=\alpha_{r}.
\]
An illustration of these posets is provided in \autoref{fig:poset-b}.

\begin{figure}[h]

\begin{center}
\begin{tikzpicture}[
scale = 0.9,
root/.style = {circle, draw=black, fill=blue!5, thick, inner sep=0pt, minimum size=7mm}
]

\node at (0, 0) [root] (a1) {$\alpha_{1,1}$};
\node at (1,-1) [root] (a2) {$\alpha_{1,2}$};
\node at (2,-2) [root] (a3) {$\alpha_{1,3}$};
\node at (3,-3) [root] (a4) {$\alpha^{\prime}_{1,3}$};
\node at (4,-4) [root] (a5) {$\alpha^{\prime}_{1,2}$};

\draw (a1) -- (a2) node [red, midway, below, sloped] {$\alpha_{2}$};
\draw (a2) -- (a3) node [red, midway, below, sloped] {$\alpha_{3}$};
\draw (a3) -- (a4) node [red, midway, below, sloped] {$\alpha_{3}$};
\draw (a4) -- (a5) node [red, midway, below, sloped] {$\alpha_{2}$};

\end{tikzpicture}
\end{center}

\caption{Poset for the $B_r$ series with $r = 3$.}
\label{fig:poset-b}

\end{figure}
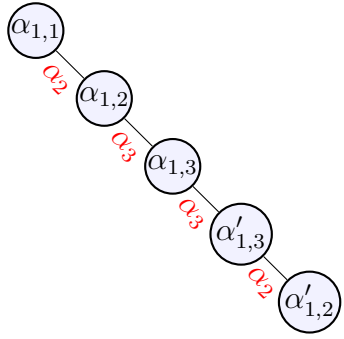

\subsection{$C_n$ series}

The positive roots of $C_r$ are given by
\[
\begin{aligned}
\alpha_{i, j} & = \sum_{k = i}^{j} \alpha_{k}, & 1 \leq i \leq j \leq r, \\
\alpha_{i, j}^{\prime} & = \sum_{k = i}^{j - 1}\alpha_{k} + \sum_{k = j}^{r - 1} 2 \alpha_{k} + \alpha_r, & 1 \leq i < j \leq r - 1.
\end{aligned}
\]
The radical roots contain $\alpha_{r}$ with multiplicity $1$. These are
\[
\{\alpha_{r,r},\cdots,\alpha_{1,r}\},\ \{\alpha_{r-1,r-1}^{\prime},\cdots,\alpha_{1,r-1}^{\prime}\},\ \cdots,\{\alpha_{2,2}^{\prime},\alpha_{1,2}^{\prime}\},\ \{\alpha_{1,1}^{\prime}\}.
\]
The poset is completely described by the following relations
\[
\alpha_{i, j}^{\prime} - \alpha_{i + 1, j}^{\prime} =\alpha_{i}, \quad
\alpha_{i, j}^{\prime} - \alpha_{i, j + 1}^{\prime} = \alpha_{j}, \quad
\alpha_{i, r - 1}^{\prime} - \alpha_{i, r} = \alpha_{r - 1}.
\]
An illustration of these posets is provided in \autoref{fig:poset-c}.

\begin{figure}[h]

\begin{center}
\begin{tikzpicture}[
scale = 0.9,
root/.style = {circle, draw=black, fill=blue!5, thick, inner sep=0pt, minimum size=7mm}
]

\node at (0, 0) [root] (a1) {$\alpha_{5,5}$};
\node at (1, 1) [root] (a2) {$\alpha_{4,5}$};
\node at (2, 2) [root] (a3) {$\alpha_{3,5}$};
\node at (3, 3) [root] (a4) {$\alpha_{2,5}$};
\node at (4, 4) [root] (a5) {$\alpha_{1,5}$};

\node at (2, 0) [root] (a6) {$\alpha^{\prime}_{4,4}$};
\node at (3, 1) [root] (a7) {$\alpha^{\prime}_{3,4}$};
\node at (4, 2) [root] (a8) {$\alpha^{\prime}_{2,4}$};
\node at (5, 3) [root] (a9) {$\alpha^{\prime}_{1,4}$};

\node at (4, 0) [root] (a10) {$\alpha^{\prime}_{3,3}$};
\node at (5, 1) [root] (a11) {$\alpha^{\prime}_{2,3}$};
\node at (6, 2) [root] (a12) {$\alpha^{\prime}_{1,3}$};

\node at (6, 0) [root] (a13) {$\alpha^{\prime}_{2,2}$};
\node at (7, 1) [root] (a14) {$\alpha^{\prime}_{1,2}$};

\node at (8, 0) [root] (a15) {$\alpha^{\prime}_{1,1}$};

\draw (a1) -- (a2) node [red, midway, below, sloped] {$\alpha_{4}$};
\draw (a2) -- (a3) node [red, midway, below, sloped] {$\alpha_{3}$};
\draw (a3) -- (a4) node [red, midway, below, sloped] {$\alpha_{2}$};
\draw (a4) -- (a5) node [red, midway, below, sloped] {$\alpha_{1}$};

\draw (a6) -- (a7) node [red, midway, below, sloped] {$\alpha_{3}$};
\draw (a7) -- (a8) node [red, midway, below, sloped] {$\alpha_{2}$};
\draw (a8) -- (a9) node [red, midway, below, sloped] {$\alpha_{1}$};

\draw (a10) -- (a11) node [red, midway, below, sloped] {$\alpha_{2}$};
\draw (a11) -- (a12) node [red, midway, below, sloped] {$\alpha_{1}$};

\draw (a13) -- (a14) node [red, midway, below, sloped] {$\alpha_{1}$};

\draw (a2) -- (a6) node [red, midway, below, sloped] {$\alpha_{4}$};
\draw (a3) -- (a7) node [red, midway, below, sloped] {$\alpha_{4}$};
\draw (a4) -- (a8) node [red, midway, below, sloped] {$\alpha_{4}$};
\draw (a5) -- (a9) node [red, midway, below, sloped] {$\alpha_{4}$};

\draw (a7) -- (a10) node [red, midway, below, sloped] {$\alpha_{3}$};
\draw (a8) -- (a11) node [red, midway, below, sloped] {$\alpha_{3}$};
\draw (a9) -- (a12) node [red, midway, below, sloped] {$\alpha_{3}$};

\draw (a11) -- (a13) node [red, midway, below, sloped] {$\alpha_{2}$};
\draw (a12) -- (a14) node [red, midway, below, sloped] {$\alpha_{2}$};

\draw (a14) -- (a15) node [red, midway, below, sloped] {$\alpha_{1}$};

\end{tikzpicture}
\end{center}

\caption{Poset for the $C_r$ series with $r = 5$.}
\label{fig:poset-c}

\end{figure}

\subsection{$D_n$ series (first case)}

The positive roots of $D_{r}$ are given by
\[
\begin{aligned}
\alpha_{i, j} & = \sum_{k = i}^{j} \alpha_{k}, & 1 \leq i \leq j \leq r - 1, \\
\alpha_{i, j}^{\prime} & = \sum_{k = i}^{j - 1}\alpha_{k} + \sum_{k = j}^{r - 2} 2 \alpha_{k} + \alpha_{r - 1} + \alpha_r, & 1 \leq i < j \leq r - 1, \\
\alpha_{i}^{\prime} & = \sum_{k = i}^{r - 2} \alpha_{k} + \alpha_{r}, & 1 \leq i \leq r - 1.
\end{aligned}
\]
The radical roots contain $\alpha_{1}$ with multiplicity $1$. These are
\[
\{\alpha_{1,1},\cdots\alpha_{1,r-1}\},\ \{\alpha_{1}^{\prime}\},\ \{\alpha_{1,r-1}^{\prime},\cdots,\alpha_{1,2}^{\prime}\}.
\]
The poset is completely described by the following relations
\[
\begin{gathered}
\alpha_{1, j} - \alpha_{1, j - 1} = \alpha_{j},\
\alpha_{1, j}^{\prime} - \alpha_{1, j + 1}^{\prime} = \alpha_{j},\\
\alpha_{1}^{\prime} - \alpha_{1, r - 2} =\alpha_{r},\
\alpha_{i, r - 1}^{\prime} - \alpha_{1, r - 1}=\alpha_{r},\
\alpha_{1, r - 1}^{\prime} - \alpha_{1}^{\prime} = \alpha_{r - 1}.
\end{gathered}
\]
An illustration of these posets is provided in \autoref{fig:poset-d1}.

\begin{figure}[h]

\begin{center}
\begin{tikzpicture}[
scale = 0.9,
root/.style = {circle, draw=black, fill=blue!5, thick, inner sep=0pt, minimum size=7mm}
]

\node at (0, 0) [root] (a1) {$\alpha_{1,1}$};
\node at (1, -1) [root] (a2) {$\alpha_{1,2}$};
\node at (2, -2) [root] (a3) {$\alpha_{1,3}$};
\node at (3, -3) [root] (a4) {$\alpha_{1,4}$};

\node at (3, -1) [root] (a5) {$\alpha^{\prime}_{1}$};

\node at (4, -2) [root] (a6) {$\alpha^{\prime}_{1,4}$};
\node at (5, -3) [root] (a7) {$\alpha^{\prime}_{1,3}$};
\node at (6, -4) [root] (a8) {$\alpha^{\prime}_{1,2}$};

\draw (a1) -- (a2) node [red, midway, below, sloped] {$\alpha_{2}$};
\draw (a2) -- (a3) node [red, midway, below, sloped] {$\alpha_{3}$};
\draw (a3) -- (a4) node [red, midway, below, sloped] {$\alpha_{4}$};

\draw (a5) -- (a6) node [red, midway, below, sloped] {$\alpha_{4}$};
\draw (a6) -- (a7) node [red, midway, below, sloped] {$\alpha_{3}$};
\draw (a7) -- (a8) node [red, midway, below, sloped] {$\alpha_{2}$};

\draw (a3) -- (a5) node [red, midway, below, sloped] {$\alpha_{5}$};
\draw (a4) -- (a6) node [red, midway, below, sloped] {$\alpha_{5}$};

\end{tikzpicture}
\end{center}

\caption{Poset for the $D_r$ series (first case) with $r = 5$.}
\label{fig:poset-d1}

\end{figure}
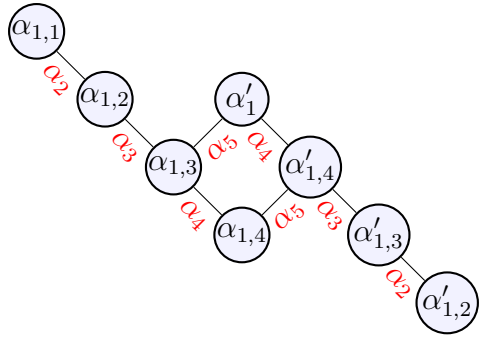

\subsection{$D_n$ series (second case)}

We keep the same notation as in the first case.
This time the radical roots contain $\alpha_{r}$ with multiplicity $1$. These are
\[
\{\alpha_{r-1}^{\prime},\cdots,\alpha_{1}^{\prime}\},\ \{\alpha_{r-2,r-1}^{\prime},\cdots,\alpha_{1,r-1}^{\prime}\},\ \cdots,\{\alpha_{2,3}^{\prime},\alpha_{1,3}^{\prime}\},\ \{\alpha_{1,2}^{\prime}\}.
\]
The poset is completely described by the following relations
\[
\alpha_{i,j}^{\prime}-\alpha_{i+1,j}^{\prime}=\alpha_{i}, \quad
\alpha_{i,j}^{\prime}-\alpha_{i,j+1}^{\prime}=\alpha_{j}, \quad
\alpha_{i,r-1}^{\prime}-\alpha_{i}^{\prime}=\alpha_{r-1}.
\]
An illustration of these posets is provided in \autoref{fig:poset-d2}.

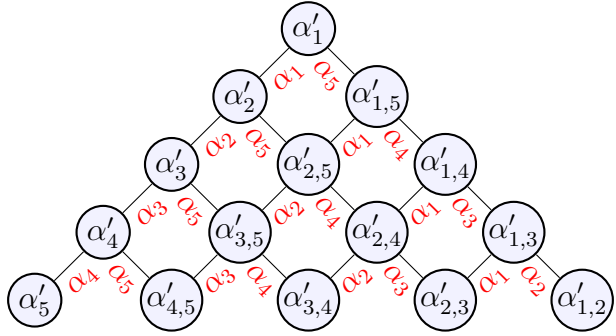
\begin{figure}[h]

\begin{center}
\begin{tikzpicture}[
scale = 0.9,
root/.style = {circle, draw=black, fill=blue!5, thick, inner sep=0pt, minimum size=7mm}
]

\node at (0, 0) [root] (a1) {$\alpha^{\prime}_{5}$};
\node at (1, 1) [root] (a2) {$\alpha^{\prime}_{4}$};
\node at (2, 2) [root] (a3) {$\alpha^{\prime}_{3}$};
\node at (3, 3) [root] (a4) {$\alpha^{\prime}_{2}$};
\node at (4, 4) [root] (a5) {$\alpha^{\prime}_{1}$};

\node at (2, 0) [root] (a6) {$\alpha^{\prime}_{4,5}$};
\node at (3, 1) [root] (a7) {$\alpha^{\prime}_{3,5}$};
\node at (4, 2) [root] (a8) {$\alpha^{\prime}_{2,5}$};
\node at (5, 3) [root] (a9) {$\alpha^{\prime}_{1,5}$};

\node at (4, 0) [root] (a10) {$\alpha^{\prime}_{3,4}$};
\node at (5, 1) [root] (a11) {$\alpha^{\prime}_{2,4}$};
\node at (6, 2) [root] (a12) {$\alpha^{\prime}_{1,4}$};

\node at (6, 0) [root] (a13) {$\alpha^{\prime}_{2,3}$};
\node at (7, 1) [root] (a14) {$\alpha^{\prime}_{1,3}$};

\node at (8, 0) [root] (a15) {$\alpha^{\prime}_{1,2}$};

\draw (a1) -- (a2) node [red, midway, below, sloped] {$\alpha_{4}$};
\draw (a2) -- (a3) node [red, midway, below, sloped] {$\alpha_{3}$};
\draw (a3) -- (a4) node [red, midway, below, sloped] {$\alpha_{2}$};
\draw (a4) -- (a5) node [red, midway, below, sloped] {$\alpha_{1}$};

\draw (a6) -- (a7) node [red, midway, below, sloped] {$\alpha_{3}$};
\draw (a7) -- (a8) node [red, midway, below, sloped] {$\alpha_{2}$};
\draw (a8) -- (a9) node [red, midway, below, sloped] {$\alpha_{1}$};

\draw (a10) -- (a11) node [red, midway, below, sloped] {$\alpha_{2}$};
\draw (a11) -- (a12) node [red, midway, below, sloped] {$\alpha_{1}$};

\draw (a13) -- (a14) node [red, midway, below, sloped] {$\alpha_{1}$};

\draw (a2) -- (a6) node [red, midway, below, sloped] {$\alpha_{5}$};
\draw (a3) -- (a7) node [red, midway, below, sloped] {$\alpha_{5}$};
\draw (a4) -- (a8) node [red, midway, below, sloped] {$\alpha_{5}$};
\draw (a5) -- (a9) node [red, midway, below, sloped] {$\alpha_{5}$};

\draw (a7) -- (a10) node [red, midway, below, sloped] {$\alpha_{4}$};
\draw (a8) -- (a11) node [red, midway, below, sloped] {$\alpha_{4}$};
\draw (a9) -- (a12) node [red, midway, below, sloped] {$\alpha_{4}$};

\draw (a11) -- (a13) node [red, midway, below, sloped] {$\alpha_{3}$};
\draw (a12) -- (a14) node [red, midway, below, sloped] {$\alpha_{3}$};

\draw (a14) -- (a15) node [red, midway, below, sloped] {$\alpha_{2}$};

\end{tikzpicture}
\end{center}

\caption{Poset for the $D_r$ series (second case) with $r = 6$.}
\label{fig:poset-d2}

\end{figure}

\subsection{$E_6$ case}

In this case we will only write the radical roots. These contain $\alpha_{5}$ with multiplicity $1$. We will use the notation $\alpha_{i, j} = \sum_{k = i}^j \alpha_k$ as in the other cases. We have
\[
\begin{gathered}
\xi_{1} = \alpha_{5,5},\
\xi_{2} = \alpha_{4,5},\
\xi_{3} = \alpha_{3,5},\
\xi_{4} = \alpha_{2,5},\
\xi_{5} = \alpha_{1,5},\
\xi_{6} = \alpha_{3,6},\
\xi_{7} = \alpha_{2,6},\
\xi_{8} = \alpha_{1,6}, \\
\xi_{9} = (0,1,2,1,1,1),\
\xi_{10} = (1,1,2,1,1,1),\
\xi_{11} = (1,2,2,1,1,1),\
\xi_{12} = (0,1,2,2,1,1),\\
\xi_{13} = (1,1,2,2,1,1),\
\xi_{14} = (1,2,2,2,1,1),\
\xi_{15} = (1,2,3,2,1,1),\
\xi_{16} = (1,2,3,2,1,2).
\end{gathered}
\]
The relations between these roots are depicted in the poset in \autoref{fig:poset-e6}.

\begin{figure}[h]

\begin{center}
\begin{tikzpicture}[
scale = 0.9,
root/.style = {circle, draw=black, fill=blue!5, thick, inner sep=0pt, minimum size=5.5mm}
]

\node at (0, 0) [root] (x1) {$\xi_{1}$};
\node at (1,-1) [root] (x2) {$\xi_{2}$};
\node at (2,-2) [root] (x3) {$\xi_{3}$};
\node at (3,-3) [root] (x4) {$\xi_{4}$};
\node at (4,-4) [root] (x5) {$\xi_{5}$};

\node at (3,-1) [root] (x6) {$\xi_{6}$};
\node at (4,-2) [root] (x7) {$\xi_{7}$};
\node at (5,-3) [root] (x8) {$\xi_{8}$};

\node at (5,-1) [root] (x9) {$\xi_{9}$};
\node at (6,-2) [root] (x10) {$\xi_{10}$};
\node at (7,-3) [root] (x13) {$\xi_{11}$};

\node at (6, 0) [root] (x11) {$\xi_{12}$};
\node at (7,-1) [root] (x12) {$\xi_{13}$};
\node at (8,-2) [root] (x14) {$\xi_{14}$};
\node at (9,-3) [root] (x15) {$\xi_{15}$};
\node at (10,-4) [root] (x16) {$\xi_{16}$};

\draw (x1) -- (x2) node [red, midway, below, sloped] {$\alpha_{4}$};
\draw (x2) -- (x3) node [red, midway, below, sloped] {$\alpha_{3}$};
\draw (x3) -- (x4) node [red, midway, below, sloped] {$\alpha_{2}$};
\draw (x4) -- (x5) node [red, midway, below, sloped] {$\alpha_{1}$};

\draw (x6) -- (x7) node [red, midway, below, sloped] {$\alpha_{2}$};
\draw (x7) -- (x8) node [red, midway, below, sloped] {$\alpha_{1}$};

\draw (x9) -- (x10) node [red, midway, below, sloped] {$\alpha_{1}$};
\draw (x10) -- (x13) node [red, midway, below, sloped] {$\alpha_{2}$};

\draw (x11) -- (x12) node [red, midway, below, sloped] {$\alpha_{1}$};
\draw (x12) -- (x14) node [red, midway, below, sloped] {$\alpha_{2}$};
\draw (x14) -- (x15) node [red, midway, below, sloped] {$\alpha_{3}$};
\draw (x15) -- (x16) node [red, midway, below, sloped] {$\alpha_{6}$};

\draw (x3) -- (x6) node [red, midway, above, sloped] {$\alpha_{6}$};
\draw (x4) -- (x7) node [red, midway, above, sloped] {$\alpha_{6}$};
\draw (x5) -- (x8) node [red, midway, above, sloped] {$\alpha_{6}$};

\draw (x7) -- (x9) node [red, midway, above, sloped] {$\alpha_{3}$};
\draw (x8) -- (x10) node [red, midway, above, sloped] {$\alpha_{3}$};

\draw (x9) -- (x11) node [red, midway, above, sloped] {$\alpha_{4}$};
\draw (x10) -- (x12) node [red, midway, above, sloped] {$\alpha_{4}$};
\draw (x13) -- (x14) node [red, midway, above, sloped] {$\alpha_{4}$};

\end{tikzpicture}
\end{center}

\caption{Poset for the $E_6$ case.}
\label{fig:poset-e6}

\end{figure}

\subsection{$E_7$ case}

We keep the same setup and notation as in the $E_6$ case. This time the radical roots contain $\alpha_{6}$ with multiplicity $1$. They are given by

\[
\begin{gathered}
\xi_{1} = \alpha_{6,6},\
\xi_{2} = \alpha_{5,6},\
\xi_{3} = \alpha_{4,6},\
\xi_{4} = \alpha_{3,6},\
\xi_{5} = \alpha_{2,6},\\
\xi_{6} = \alpha_{1,6},\
\xi_{7} = \alpha_{3,7},\
\xi_{8} = \alpha_{2,7},\
\xi_{9} = \alpha_{1,7},\\
\xi_{10} = (0,1,2,1,1,1,1),\
\xi_{11} = (1,1,2,1,1,1,1),\
\xi_{12} = (1,2,2,1,1,1,1),\\
\xi_{13} = (0,1,2,2,1,1,1),\
\xi_{14} = (1,1,2,2,1,1,1),\
\xi_{15} = (1,2,2,2,1,1,1),\\
\xi_{16} = (1,2,3,2,1,1,1),\
\xi_{17} = (1,2,3,2,1,1,2),\
\xi_{18} = (0,1,2,2,2,1,1),\\
\xi_{19} = (1,1,2,2,2,1,1),\
\xi_{20} = (1,2,2,2,2,1,1),\
\xi_{21} = (1,2,3,2,2,1,1),\\
\xi_{22} = (1,2,3,2,2,1,2),\
\xi_{23} = (1,2,3,3,2,1,1),\
\xi_{24} = (1,2,3,3,2,1,2),\\
\xi_{25} = (1,2,4,3,2,1,2),\
\xi_{26} = (1,3,4,3,2,1,2),\
\xi_{27} = (2,3,4,3,2,1,2).
\end{gathered}
\]
The relations between these roots are depicted in the poset in \autoref{fig:poset-e6}.

\begin{figure}[h]

\begin{center}
\begin{tikzpicture}[
scale = 0.9,
root/.style = {circle, draw=black, fill=blue!5, thick, inner sep=0pt, minimum size=5.5mm}
]

\node at (0, 0) [root] (x1) {$\xi_{1}$};
\node at (1,-1) [root] (x2) {$\xi_{2}$};
\node at (2,-2) [root] (x3) {$\xi_{3}$};
\node at (3,-3) [root] (x4) {$\xi_{4}$};
\node at (4,-4) [root] (x5) {$\xi_{5}$};
\node at (5,-5) [root] (x6) {$\xi_{6}$};

\node at (4,-2) [root] (x7) {$\xi_{7}$};
\node at (5,-3) [root] (x8) {$\xi_{8}$};
\node at (6,-4) [root] (x9) {$\xi_{9}$};

\node at (6,-2) [root] (x10) {$\xi_{10}$};
\node at (7,-3) [root] (x11) {$\xi_{11}$};
\node at (8,-4) [root] (x12) {$\xi_{12}$};

\node at (7,-1) [root] (x13) {$\xi_{13}$};
\node at (8,-2) [root] (x14) {$\xi_{14}$};
\node at (9,-3) [root] (x15) {$\xi_{15}$};
\node at (10,-4) [root] (x16) {$\xi_{16}$};
\node at (11,-5) [root] (x17) {$\xi_{17}$};

\node at (8,0) [root] (x18) {$\xi_{18}$};
\node at (9,-1) [root] (x19) {$\xi_{19}$};
\node at (10,-2) [root] (x20) {$\xi_{20}$};
\node at (11,-3) [root] (x21) {$\xi_{21}$};
\node at (12,-4) [root] (x22) {$\xi_{22}$};

\node at (12,-2) [root] (x23) {$\xi_{23}$};
\node at (13,-3) [root] (x24) {$\xi_{24}$};

\node at (14,-2) [root] (x25) {$\xi_{25}$};
\node at (15,-1) [root] (x26) {$\xi_{26}$};
\node at (16, 0) [root] (x27) {$\xi_{27}$};

\draw (x1) -- (x2) node [red, midway, below, sloped] {$\alpha_{5}$};
\draw (x2) -- (x3) node [red, midway, below, sloped] {$\alpha_{4}$};
\draw (x3) -- (x4) node [red, midway, below, sloped] {$\alpha_{3}$};
\draw (x4) -- (x5) node [red, midway, below, sloped] {$\alpha_{2}$};
\draw (x5) -- (x6) node [red, midway, below, sloped] {$\alpha_{1}$};

\draw (x7) -- (x8) node [red, midway, below, sloped] {$\alpha_{2}$};
\draw (x8) -- (x9) node [red, midway, below, sloped] {$\alpha_{1}$};

\draw (x10) -- (x11) node [red, midway, below, sloped] {$\alpha_{1}$};
\draw (x11) -- (x12) node [red, midway, below, sloped] {$\alpha_{2}$};

\draw (x13) -- (x14) node [red, midway, below, sloped] {$\alpha_{1}$};
\draw (x14) -- (x15) node [red, midway, below, sloped] {$\alpha_{2}$};
\draw (x15) -- (x16) node [red, midway, below, sloped] {$\alpha_{3}$};
\draw (x16) -- (x17) node [red, midway, below, sloped] {$\alpha_{7}$};

\draw (x18) -- (x19) node [red, midway, below, sloped] {$\alpha_{1}$};
\draw (x19) -- (x20) node [red, midway, below, sloped] {$\alpha_{2}$};
\draw (x20) -- (x21) node [red, midway, below, sloped] {$\alpha_{3}$};
\draw (x21) -- (x22) node [red, midway, below, sloped] {$\alpha_{7}$};

\draw (x23) -- (x24) node [red, midway, below, sloped] {$\alpha_{7}$};

\draw (x4) -- (x7) node [red, midway, below, sloped] {$\alpha_{7}$};
\draw (x5) -- (x8) node [red, midway, below, sloped] {$\alpha_{7}$};
\draw (x6) -- (x9) node [red, midway, below, sloped] {$\alpha_{7}$};

\draw (x8) -- (x10) node [red, midway, below, sloped] {$\alpha_{3}$};
\draw (x9) -- (x11) node [red, midway, below, sloped] {$\alpha_{3}$};

\draw (x10) -- (x13) node [red, midway, below, sloped] {$\alpha_{4}$};
\draw (x11) -- (x14) node [red, midway, below, sloped] {$\alpha_{4}$};
\draw (x12) -- (x15) node [red, midway, below, sloped] {$\alpha_{4}$};

\draw (x13) -- (x18) node [red, midway, below, sloped] {$\alpha_{5}$};
\draw (x14) -- (x19) node [red, midway, below, sloped] {$\alpha_{5}$};
\draw (x15) -- (x20) node [red, midway, below, sloped] {$\alpha_{5}$};
\draw (x16) -- (x21) node [red, midway, below, sloped] {$\alpha_{5}$};
\draw (x17) -- (x22) node [red, midway, below, sloped] {$\alpha_{5}$};

\draw (x21) -- (x23) node [red, midway, below, sloped] {$\alpha_{4}$};
\draw (x22) -- (x24) node [red, midway, below, sloped] {$\alpha_{4}$};

\draw (x24) -- (x25) node [red, midway, below, sloped] {$\alpha_{3}$};
\draw (x25) -- (x26) node [red, midway, below, sloped] {$\alpha_{2}$};
\draw (x26) -- (x27) node [red, midway, below, sloped] {$\alpha_{1}$};

\end{tikzpicture}
\end{center}

\caption{Poset for the $E_7$ case.}
\label{fig:poset-e7}

\end{figure}
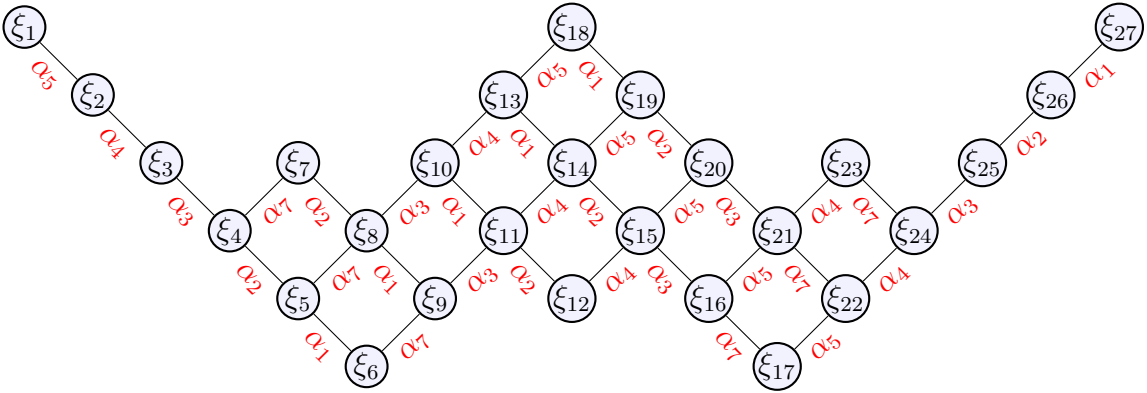

\section{Enumeration of the posets}
\label{sec:enumeration}

In this section we will give a uniform construction of the posets of radical roots. This setup will allow us to easily check certain properties that they satisfy.

\subsection{Construction}

The construction is as follows. Start with a box of size $A\times B$ and vertices $\{v_{i, j}\}$ with $1 \leq i \leq A$ and $1 \leq j \leq B$ and edges connecting $v_{i, j}$ to $v_{i + 1, j}$ and $v_{i, j + 1}$. We require that the label of the edge connecting $v_{i, j}$ to $v_{i + 1,j}$ depends only on $i$ and we denote it by $a_{i}$.
Similarly the label of the edge connecting $v_{i, j}$ to $v_{i, j + 1}$ depends only on $j$ and we denote it by $d_{j}$.
Here the letters $a_i$ and $d_j$ mean anti-diagonal and diagonal, respectively.
An illustration of this setup is given in \autoref{fig:construction}.
From this starting configuration we are allowed to remove some vertices and the corresponding edges.
Comparison with our previous explicit description shows that all the posets of radical roots can be realized in this way (for certain choices of vertices to be removed).
Of course not all graphs obtained in this way correspond to posets of radical roots.
This will not be important in the following.

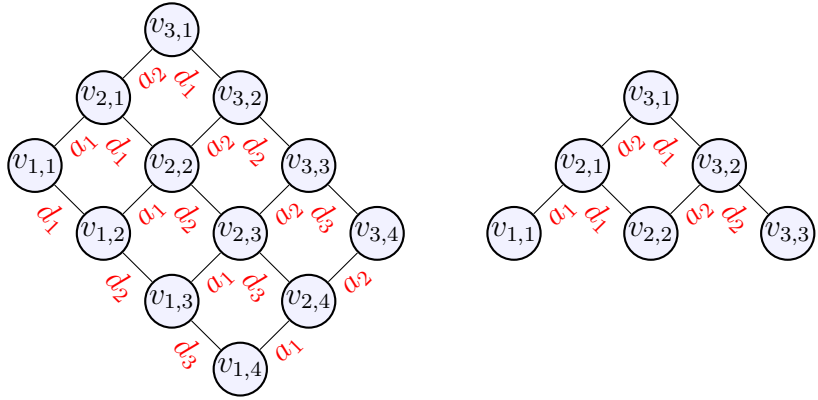
\begin{figure}[h]
\begin{center}
\begin{tikzpicture}[
scale = 0.9,
root/.style = {circle, draw=black, fill=blue!5, thick, inner sep=0pt, minimum size=7mm}
]

\node at (0,  0) [root] (v11) {$v_{1,1}$};
\node at (1, -1) [root] (v12) {$v_{1,2}$};
\node at (2, -2) [root] (v13) {$v_{1,3}$};
\node at (3, -3) [root] (v14) {$v_{1,4}$};

\node at (1,  1) [root] (v21) {$v_{2,1}$};
\node at (2,  0) [root] (v22) {$v_{2,2}$};
\node at (3, -1) [root] (v23) {$v_{2,3}$};
\node at (4, -2) [root] (v24) {$v_{2,4}$};

\node at (2,  2) [root] (v31) {$v_{3,1}$};
\node at (3,  1) [root] (v32) {$v_{3,2}$};
\node at (4,  0) [root] (v33) {$v_{3,3}$};
\node at (5, -1) [root] (v34) {$v_{3,4}$};

\draw (v11) -- (v12) node [red, midway, below, sloped] {$d_1$};
\draw (v12) -- (v13) node [red, midway, below, sloped] {$d_2$};
\draw (v13) -- (v14) node [red, midway, below, sloped] {$d_3$};

\draw (v21) -- (v22) node [red, midway, below, sloped] {$d_1$};
\draw (v22) -- (v23) node [red, midway, below, sloped] {$d_2$};
\draw (v23) -- (v24) node [red, midway, below, sloped] {$d_3$};

\draw (v31) -- (v32) node [red, midway, below, sloped] {$d_1$};
\draw (v32) -- (v33) node [red, midway, below, sloped] {$d_2$};
\draw (v33) -- (v34) node [red, midway, below, sloped] {$d_3$};

\draw (v11) -- (v21) node [red, midway, below, sloped] {$a_1$};
\draw (v21) -- (v31) node [red, midway, below, sloped] {$a_2$};

\draw (v12) -- (v22) node [red, midway, below, sloped] {$a_1$};
\draw (v22) -- (v32) node [red, midway, below, sloped] {$a_2$};

\draw (v13) -- (v23) node [red, midway, below, sloped] {$a_1$};
\draw (v23) -- (v33) node [red, midway, below, sloped] {$a_2$};

\draw (v14) -- (v24) node [red, midway, below, sloped] {$a_1$};
\draw (v24) -- (v34) node [red, midway, below, sloped] {$a_2$};


\node at (7, -1) [root] (x11) {$v_{1,1}$};

\node at (8,  0) [root] (x21) {$v_{2,1}$};
\node at (9, -1) [root] (x22) {$v_{2,2}$};

\node at (9,  1) [root] (x31) {$v_{3,1}$};
\node at (10, 0) [root] (x32) {$v_{3,2}$};
\node at (11,-1) [root] (x33) {$v_{3,3}$};

\draw (x21) -- (x22) node [red, midway, below, sloped] {$d_1$};

\draw (x31) -- (x32) node [red, midway, below, sloped] {$d_1$};
\draw (x32) -- (x33) node [red, midway, below, sloped] {$d_2$};

\draw (x11) -- (x21) node [red, midway, below, sloped] {$a_1$};
\draw (x21) -- (x31) node [red, midway, below, sloped] {$a_2$};

\draw (x22) -- (x32) node [red, midway, below, sloped] {$a_2$};

\end{tikzpicture}
\end{center}

\caption{Left: a starting box of size $3 \times 4$. Right: one of the posets realized in this way (corresponding for example to the $C_r$ series).}
\label{fig:construction}

\end{figure}

The poset structure is obtained in the following way. We assigning to each vertex $v_{i, j}$ a radical root $\xi(v_{i, j})$ and to the labels $a_{i}$ and $d_{j}$ the simple roots such that
\[
\xi(v_{i + 1, j}) - \xi(v_{i, j}) = a_{i},\quad
\xi(v_{i,j + 1}) - \xi(v_{i, j}) = d_{j}.
\]
Observe that in this construction the radical root $\alpha_{t}$ corresponds to the vertex $v_{1, 1}$, while the highest root (which is always a radical root) corresponds to the vertex
$v_{A, B}$.

This specific enumeration of the vertices gives a simple way to tell when two radical roots are comparable: indeed, we have that $\xi(v_{a, b}) \preceq \xi(v_{c, d})$ if and only if $a \leq c$ and $b \leq d$.
In all other cases the two radical roots are not comparable.
This claim easily follows from the rules given previously defining the labels $\{a_i\}$ and $\{d_j\}$.

\begin{notation}
Let $F(i)$ be the number such that $v_{i, F(i)}$ is the \emph{first} element along the diagonal $i$. Similarly let $L(i)$ be the number such that $v_{i, L(i)}$ is the \emph{last} element along the diagonal $i$.
\end{notation}

These functions are needed to express some further properties of the posets in this realization. They can be obtained by inspection of the corresponding posets or by general arguments related to root systems.
The properties that we need are the following.

\begin{itemize}
\item For every diagonal $i$ we have all the vertices $\{v_{i, j}\}$ with $F(i) \leq j \leq L(i)$.
\item We have $F(i) \leq F(k)$ for $i < k$. Otherwise there would be a positive root which could not be obtained from a root of lower height by adding a simple root.
\item We have $L(i) \leq L(k)$ for $i < k$. Otherwise there would be a positive root which could not be obtained by subtracting simple roots from the highest root.
\end{itemize}

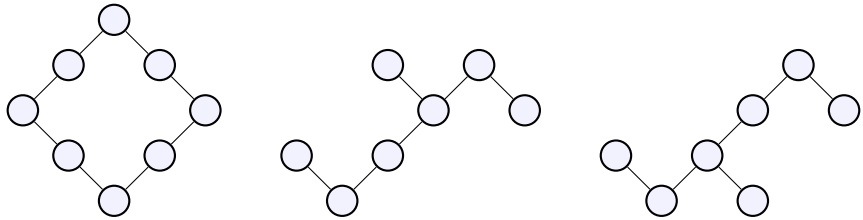
\begin{figure}[h]

\begin{center}
\begin{tikzpicture}[
scale = 0.6,
root/.style = {circle, draw=black, fill=blue!5, thick, inner sep=0pt, minimum size=4mm}
]

\node at ( 0,  0) [root] (a1) {$$};
\node at ( 1,  1) [root] (a2) {$$};
\node at ( 2,  2) [root] (a3) {$$};
\node at ( 3,  1) [root] (a4) {$$};
\node at ( 4,  0) [root] (a5) {$$};
\node at ( 3, -1) [root] (a6) {$$};
\node at ( 2, -2) [root] (a7) {$$};
\node at ( 1, -1) [root] (a8) {$$};

\draw (a1) -- (a2) node [red, midway, below, sloped] {$$};
\draw (a2) -- (a3) node [red, midway, below, sloped] {$$};
\draw (a3) -- (a4) node [red, midway, below, sloped] {$$};
\draw (a4) -- (a5) node [red, midway, below, sloped] {$$};
\draw (a5) -- (a6) node [red, midway, below, sloped] {$$};
\draw (a6) -- (a7) node [red, midway, below, sloped] {$$};
\draw (a7) -- (a8) node [red, midway, below, sloped] {$$};
\draw (a8) -- (a1) node [red, midway, below, sloped] {$$};

\node at ( 6, -1) [root] (b1) {$$};
\node at ( 7, -2) [root] (b2) {$$};
\node at ( 8, -1) [root] (b3) {$$};
\node at ( 9,  0) [root] (b4) {$$};
\node at ( 8,  1) [root] (b5) {$$};
\node at (10,  1) [root] (b6) {$$};
\node at (11,  0) [root] (b7) {$$};

\draw (b1) -- (b2) node [red, midway, below, sloped] {$$};
\draw (b2) -- (b3) node [red, midway, below, sloped] {$$};
\draw (b3) -- (b4) node [red, midway, below, sloped] {$$};
\draw (b4) -- (b5) node [red, midway, below, sloped] {$$};
\draw (b4) -- (b6) node [red, midway, below, sloped] {$$};
\draw (b6) -- (b7) node [red, midway, below, sloped] {$$};

\node at (13, -1) [root] (c1) {$$};
\node at (14, -2) [root] (c2) {$$};
\node at (15, -1) [root] (c3) {$$};
\node at (16, -2) [root] (c4) {$$};
\node at (16,  0) [root] (c5) {$$};
\node at (17,  1) [root] (c6) {$$};
\node at (18,  0) [root] (c7) {$$};

\draw (c1) -- (c2) node [red, midway, below, sloped] {$$};
\draw (c2) -- (c3) node [red, midway, below, sloped] {$$};
\draw (c3) -- (c4) node [red, midway, below, sloped] {$$};
\draw (c3) -- (c5) node [red, midway, below, sloped] {$$};
\draw (c5) -- (c6) node [red, midway, below, sloped] {$$};
\draw (c6) -- (c7) node [red, midway, below, sloped] {$$};

\end{tikzpicture}
\end{center}

\caption{Illustration of the forbidden configurations.}

\end{figure}

\subsection{Another property of the posets}

In this subsection we describe one less obvious property satisfied by the posets of radical roots.
Although it is not apparent at this stage why we would need such a property, it will become clear later on.

Recall that along the diagonal $i$ we have the vertices $\{v_{i, F(i)}, \cdots, v_{i, L(i)}\}$.
These are connected by the labels $\{d_{F(i)}, \cdots, d_{L(i) - 1}\}$ via the relations $\xi(v_{i, j + 1}) - \xi(v_{i, j}) = d_{j}$. Also recall that for any $i$ we have $L(i) \leq L(i + 1)$. We define the "line" $L_i$ by
\[
L_{i} = \{d_{F(i + 1)}, \cdots, d_{L(i) - 1}\}.
\]
It is a subset of all the labels which occur along the diagonal $i + 1$ (not the diagonal $i$, as the notation might suggest). This definition is better understood by looking at \autoref{fig:lines}.

\begin{figure}[h]

\begin{center}
\begin{tikzpicture}[
scale = 0.6,
root/.style = {circle, draw=black, fill=blue!5, thick, inner sep=0pt, minimum size = 4mm},
redroot/.style = {circle, draw=black, fill = red!40, thick, inner sep=0pt, minimum size = 4mm},
greenroot/.style = {circle, draw = black, fill = green!50, thick, inner sep=0pt, minimum size = 4mm}
]

\node at (0, 0) [root] (a1) {};
\node at (1, 1) [root] (a2) {};
\node at (2, 2) [root] (a3) {};
\node at (3, 3) [greenroot] (a4) {};

\node at (2, 0) [root] (a6) {};
\node at (3, 1) [root] (a7) {};
\node at (4, 2) [greenroot] (a8) {};

\node at (4, 0) [redroot] (a10) {};
\node at (5, 1) [greenroot] (a11) {};

\node at (6, 0) [root] (a13) {};

\draw (a1) -- (a2) node [red, midway, below, sloped] {};
\draw (a2) -- (a3) node [red, midway, below, sloped] {};
\draw (a3) -- (a4) node [red, midway, below, sloped] {};

\draw (a6) -- (a7) node [red, midway, below, sloped] {};
\draw (a7) -- (a8) node [red, midway, below, sloped] {};

\draw (a10) -- (a11) node [red, midway, below, sloped] {};

\draw (a2) -- (a6) node [red, midway, below, sloped] {};
\draw (a3) -- (a7) node [red, midway, below, sloped] {};
\draw (a4) -- (a8) node [red, midway, below, sloped] {};

\draw (a7) -- (a10) node [red, midway, below, sloped] {};
\draw (a8) -- (a11) node [red, midway, below, sloped] {};

\draw (a11) -- (a13) node [red, midway, below, sloped] {};

\end{tikzpicture}
\end{center}

\caption{In red the vertex $v_{i, L(i)}$, in green the vertices connected by $L_i$.}
\label{fig:lines}

\end{figure}
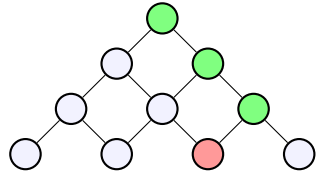

It is not difficult to see that the label $a_i$ can \emph{not} appear in $L_i$. If this were the case, we would have a root on the diagonal $i + 1$ identical to a root on the diagonal $i$, which is impossible.

The question is then whether the label $a_i$ can appear in the line $L_j$ for $i \neq j$.
To answer this question we write down explicitely, in the following list, the values of $a_i$ in increasing order and the corresponding sets $L_i$ for all the posets.

\begin{itemize}

\item
For the $A_r$ series the values of $a_i$ are given by $\{ \alpha_{t + 1}, \cdots, \alpha_r \}$. For all of them the corresponding set is $\{ \alpha_{1}, \cdots, \alpha_t \}$.

\item For the $B_r$ series we only have one diagonal, so there is nothing to check.

\item
For the $C_r$ series the values of $a_i$ are given by $\{ \alpha_{r - 1}, \cdots, \alpha_1 \}$.
Corresponding to $\alpha_k$ we have the set $\{\alpha_{k + 1}, \cdots, \alpha_{r - 1}\}$, which is empty for $\alpha_{r - 1}$.

\item
For the $D_r$ series (first case) we only have two diagonals, so there is nothing to check.

\item
For the $D_r$ series (second case) the values of $a_i$ are given by $\{ \alpha_{r - 2}, \cdots, \alpha_1 \}$.
Corresponding to $\alpha_k$ we have the set $\{\alpha_{k + 2}, \cdots, \alpha_{r - 1}\}$, which is empty for $\alpha_{r - 2}$.

\item The $E_6$ and $E_7$ cases are irregular. We list them in the tables below, with $E_6$ to the left and $E_7$ to the right. We also skip values $a_i$ such that $L_i$ is empty.

\begin{center}
\begin{tabular}{|c|c|}
\hline 
$a_{i}$ & $L_{i}$
\tabularnewline
\hline 
\hline 
$\alpha_{6}$ & $\{\alpha_{1},\alpha_{2}\}$
\tabularnewline
\hline 
$\alpha_{3}$ & $\{\alpha_{1}\}$
\tabularnewline
\hline 
$\alpha_{4}$ & $\{\alpha_{1},\alpha_{2}\}$
\tabularnewline
\hline 
\end{tabular}
\qquad
\begin{tabular}{|c|c|}
\hline 
$a_{i}$ & $L_{i}$
\tabularnewline
\hline 
\hline 
$\alpha_{7}$ & $\{\alpha_{1}, \alpha_{2}\}$
\tabularnewline
\hline 
$\alpha_{3}$ & $\{\alpha_{1}\}$
\tabularnewline
\hline 
$\alpha_{4}$ & $\{\alpha_{1}, \alpha_{2}\}$
\tabularnewline
\hline 
$\alpha_{5}$ & $\{\alpha_{1}, \alpha_{2},\alpha_{3}, \alpha_{7}\}$
\tabularnewline
\hline 
$\alpha_{4}$ & $\{\alpha_{7}\}$
\tabularnewline
\hline 
\end{tabular}
\end{center}

\end{itemize}

\smallskip

\begin{lemma}
\label{lem:lines}
Suppose that $a_i \in L_j$ for some $i$ and $j$. Then we must have $i < j$.
\end{lemma}

\begin{proof}
This is easily checked using the list given above.
\end{proof}

\begin{remark}
Notice that \autoref{lem:lines} is not true if, instead of $L_i$, we consider all the labels along the diagonal $i + 1$. This can be seen for example in the poset of the $C_r$ series.
\end{remark}

\begin{remark}
For the cases $A_r$, $C_r$ and $D_r$ it is possible to give a more conceptual proof of \autoref{lem:lines} using the fact that $F(i) = 1$ for all $i$.
However we were not able to find such a proof for the irregular cases $E_6$ and $E_7$, hence the explicit description.
\end{remark}

\section{The main proof}
\label{sec:proof}

Recall that with our enumeration of the vertices we have $\xi(v_{a, b}) \preceq \xi(v_{c, d})$ if and only if $a \leq c$ and $b \leq d$.
All the other cases are not comparable.
We now fix a specific total order which is compatible with this partial order, as discussed in \autoref{sec:convex-ord}.

\begin{definition}
We define the total order $<$ by $\xi(v_{a, b}) < \xi(v_{c, d})$ if $a < c$ or $a = c$ and $b < d$.
\end{definition}

In \autoref{fig:poset-e6} and \autoref{fig:poset-e7} one can see how this total order looks graphically, for the posets of $E_6$ and $E_7$.
In the next lemma we associate two radical roots $\xi_S$ and $\xi_B$ to any radical root $\xi$. The subscripts correspond to "small" and "big", respectively.

\begin{lemma}
Let $\xi = \xi(v_{a, b})$ with $1 < a < A$ and define the radical roots $\xi_S = \xi(v_{a + 1, F(a + 1)})$ and $\xi_B = \xi(v_{a - 1, L(a - 1)})$. They satisfy the following properties:

\begin{itemize}
\item
$\xi_S \preceq \xi^\prime$ for any $\xi^\prime$ such that $\xi^\prime > \xi$ and $\xi^\prime \not \succ \xi$,
\item
$\xi_B \succeq \xi^\prime$ for any $\xi^\prime$ such that $\xi^\prime < \xi$ and $\xi^\prime \not \prec \xi$.
\end{itemize}
\end{lemma}

\begin{proof}
Let $\xi^\prime = \xi(v_{c, d})$ be such that $\xi^\prime > \xi$. First consider the case $a = c$ and $b < d$. It follows that $\xi(v_{a, b}) \prec \xi(v_{a, d})$, so that we need not consider this case. Now consider the case $a < c$. Recall that $F(a) < F(c)$ if $a < c$. Then we have $\xi(v_{a + 1, F(a + 1)}) \preceq \xi(v_{c, d})$.

Similarly let $\xi^\prime = \xi(v_{c, d})$ be such that $\xi^\prime < \xi$. First consider the case $a = c$ and $b > d$. It follows that $\xi(v_{a, b}) \succ \xi(v_{a, d})$, so that we need not consider this case. Now consider the case $a > c$. Recall that $L(a) > L(c)$ if $a > c$. Then we have $\xi(v_{a - 1, L(a - 1)}) \succeq \xi(v_{c, d})$.
\end{proof}

The roots $\xi_S$ and $\xi_B$ are depicted in \autoref{fig:strategy-proof}. This figure also illustrates the strategy of the proof of the next proposition, which provides an answer to \autoref{prob:equality}.

\begin{figure}[h]

\begin{center}
\begin{tikzpicture}[
scale = 0.8,
root/.style = {circle, draw=black, fill=blue!5, thick, inner sep=0pt, minimum size=6mm},
redroot/.style = {circle, draw=black, fill=red!15, thick, inner sep=0pt, minimum size=6mm},
greenroot/.style = {circle, draw=black, fill=green!20, thick, inner sep=0pt, minimum size=6mm}
]

\node at (0, 0) [root] (a1) {};
\node at (1, 1) [root] (a2) {};
\node at (2, 2) [redroot] (a3) {$\xi_{j}$};
\node at (3, 3) [root] (a4) {};
\node at (4, 4) [greenroot] (a5) {$\xi_S$};

\node at (2, 0) [greenroot] (a6) {$\xi_B$};
\node at (3, 1) [root] (a7) {};
\node at (4, 2) [root] (a8) {};
\node at (5, 3) [root] (a9) {};

\node at (4, 0) [root] (a10) {};
\node at (5, 1) [redroot] (a11) {$\xi_{i}$};
\node at (6, 2) [root] (a12) {};

\node at (6, 0) [root] (a13) {};
\node at (7, 1) [root] (a14) {};

\node at (8, 0) [root] (a15) {};

\draw (a1) -- (a2) node [red, midway, below, sloped] {};
\draw (a2) -- (a3) node [red, midway, below, sloped] {};
\draw (a3) -- (a4) node [red, midway, below, sloped] {};
\draw (a4) -- (a5) node [red, midway, below, sloped] {};

\draw [red, line width=1.6pt] (a6) -- (a7) node {};
\draw (a7) -- (a8) node {};
\draw (a8) -- (a9) node {};

\draw (a10) -- (a11) node {};
\draw [red, line width=1.6pt] (a11) -- (a12) node {};

\draw (a13) -- (a14) node {};

\draw (a2) -- (a6) node {};
\draw [red, line width=1.6pt] (a3) -- (a7) node {};
\draw (a4) -- (a8) node {};
\draw [red, line width=1.6pt] (a5) -- (a9) node {};

\draw (a7) -- (a10) node {};
\draw (a8) -- (a11) node {};
\draw [red, line width=1.6pt] (a9) -- (a12) node {};

\draw (a11) -- (a13) node {};
\draw (a12) -- (a14) node {};

\draw (a14) -- (a15) node {};

\end{tikzpicture}
\end{center}

\caption{We associate $\xi_S$ to $\xi_i$ and $\xi_B$ to $\xi_j$ (connected by red lines).}
\label{fig:strategy-proof}

\end{figure}
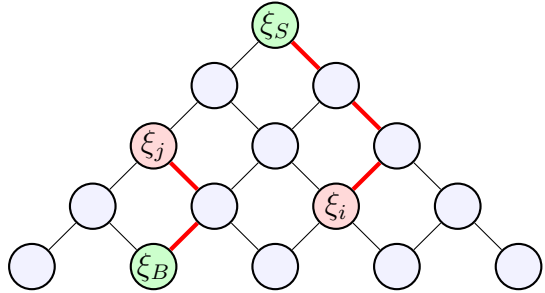

\begin{proposition}
\label{pro:main-proof}
Let $i > j$. Suppose we have the equality
\[
\xi_{i} - \xi_{j} = \sum_{\xi > \xi_{i}} c_{\xi} \xi - \sum_{\xi < \xi_{j}} c_{\xi} \xi + \gamma.
\]
Then necessarily we must have $c_{\xi} = 0$ for all $\xi$.
\end{proposition}

\begin{proof}
According to our setup, we must have $\xi_{i} = \xi(v_{a, b})$ and $\xi_{j} = \xi(v_{c, d})$
with $a > c$ or $a = c$ and $b > d$.
We can get rid of all $\xi$ such that $\xi \succ \xi_{i}$ and all $\xi$ such that $\xi \prec \xi_{j}$, using \autoref{lem:eli-ord}.
Notice that, in the first case, we can get rid of the roots $\xi(v_{a, k})$ with $k > b$, while, in the second case, we can get rid of the roots $\xi(v_{c, k})$ with $k < d$.
This takes care of the cases where either $\xi_{i}$ is on the last diagonal or $\xi_{j}$
is on the first diagonal. Indeed, in both cases one of the two sums is empty, which in turn implies that no radical roots can appear.

Now we consider the general case.
To the root $\xi_{i} = \xi(v_{a, b})$ we associate $\xi_S = \xi(v_{a + 1, F(a + 1)})$, while to the root $\xi_{j} = \xi(v_{c, d})$ we associate $\xi_B = \xi(v_{c - 1, L(c - 1)})$.
Consider the equality
\[
\xi_S - \xi_B = (\xi_{i} - \xi_{j}) + (\xi_S - \xi_{i}) + (\xi_{j} - \xi_B).
\]
We want to show that the sum $(\xi_S - \xi_{i}) + (\xi_{j} - \xi_B)$ contains at least one simple root with positive coefficient.
If this is the case, then we can apply \autoref{lem:eli-diff} to finish the proof. Indeed, the radical roots $\xi_S$ and $\xi_B$ would satisfy all the assumptions of that lemma.

We can reduce this check to some special cases.
Let us define $\chi _B = \xi(v_{a, L(a)})$ and $\chi_S = \xi(v_{c, F(c)})$.
Then we rewrite the previous equality as follows
\[
\begin{split}
\xi_S - \xi_B & = (\xi_{i} - \xi_{j}) + (\xi_S - \chi_B) + (\chi_S - \xi_B) \\
& + (\chi_B - \xi_i) + (\xi_{j} - \chi_S).
\end{split}
\]
It is immediate to see that $\chi_B \succeq \xi_i$ and $\xi_j \succeq \chi_S$.
Therefore it is enough to check that $(\xi_S - \chi_B) + (\chi_S - \xi_B)$ contains at least one simple root with positive coefficient.

To proceed we write down explicitely the two differences
\[
\begin{split}
\xi(v_{a + 1, F(a + 1)}) - \xi(v_{a, L(a)})
& = a_a - \sum_{k = F(a + 1)}^{L(a) - 1} d_k,\\
\xi(v_{c, F(c)}) - \xi(v_{c - 1, L(c - 1)})
& = a_{c - 1} - \sum_{k = F(c)}^{L(c - 1) - 1} d_k.
\end{split}
\]
We can rewrite these in terms of the sets $L_i$ introduced previously as
\[
\xi_S - \chi_B = a_a - \sum_{\alpha \in L_a} \alpha, \quad
\chi_S - \xi_B = a_{c - 1} - \sum_{\alpha \in L_{c - 1}} \alpha.
\]
Using \autoref{lem:lines} we have that $a_i \in L_j$ implies $i < j$. Since we have $a \geq c$, it follows that $a_a$ does not belong to the sets $L_a$ and $L_{c - 1}$.
Therefore the simple root $a_a$ appears in the sum $(\xi_S - \chi_B) + (\chi_S - \xi_B)$ with positive coefficient. This concludes the proof.

\end{proof}

Therefore, for the particular order that we have chosen in this section, we have a positive answer to \autoref{prob:equality} for all $i > j$. This concludes the proof of \autoref{thm:main-thm}.

\section{Dependence on the reduced decomposition}
\label{sec:counter}

In this section we show that \autoref{thm:main-thm} does not hold for all choices of reduced decomposition. We present a simple counterexample for the cominuscole parabolic corresponding to $C_{3}$. Many other similar counterexamples can be easily exhibited.

\begin{proposition}
\autoref{thm:main-thm} does not hold for all reduced decompositions.
\end{proposition}

\begin{proof}
Consider the cominuscole parabolic corresponding to $C_{3}$. It is defined by the simple root $\alpha_t = \alpha_3$. Let us choose the reduced decomposition
\[
w_0 = s_3 (s_2 s_3 s_2) (s_1 s_2 s_3 s_2 s_1).
\]
Correspondingly, we have the convex order on the positive roots given by
\begin{gather*}
\alpha_{3} < \alpha_{2} + \alpha_{3} < 2\alpha_{2} + \alpha_{3} < \alpha_{2} < \alpha_{1} + 2\alpha_{2} + \alpha_{3} \\
< \alpha_{1} + \alpha_{2} + \alpha_{3} < 2\alpha_{1} + 2\alpha_{2} + \alpha_{3} < \alpha_{1} + \alpha_{2} < \alpha_{1}.
\end{gather*}
Notice that, for this particular order, we do not have the property $\xi < \alpha$ for all radical roots $\xi$ and non-radical roots $\alpha$.
Focusing on the radical roots, we have the enumeration
\begin{gather*}
\xi_{1} = \alpha_{3}, \
\xi_{2} = \alpha_{2} + \alpha_{3}, \
\xi_{3} = 2\alpha_{2} + \alpha_{3}, \
\xi_{4} = \alpha_{1} + 2\alpha_{2} + \alpha_{3}, \\
\xi_{5} = \alpha_{1} + \alpha_{2} + \alpha_{3}, \
\xi_{6} = 2\alpha_{1} + 2\alpha_{2} + \alpha_{3}.
\end{gather*}
Then we simply need to observe that we have
\[
\xi_{4} - \xi_{2} = \xi_{5} - \xi_{1}.
\]
Therefore \autoref{prob:equality} has a negative answer in this case.
Moreover, it can be checked by explicit computations that \autoref{thm:main-thm} does not hold for this choice.
\end{proof}

The example described above is depicted in \autoref{fig:counter}.

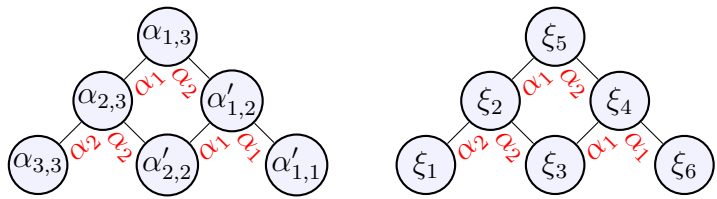
\begin{figure}[h]

\begin{center}
\begin{tikzpicture}[
scale = 0.85,
root/.style = {circle, draw=black, fill=blue!5, thick, inner sep=0pt, minimum size=7.7mm}
]

\node at (0, 0) [root] (a1) {$\alpha_{3,3}$};
\node at (1, 1) [root] (a2) {$\alpha_{2,3}$};
\node at (2, 2) [root] (a3) {$\alpha_{1,3}$};

\node at (2, 0) [root] (a4) {$\alpha^{\prime}_{2,2}$};
\node at (3, 1) [root] (a5) {$\alpha^{\prime}_{1,2}$};

\node at (4, 0) [root] (a6) {$\alpha^{\prime}_{1,1}$};

\draw (a1) -- (a2) node [red, midway, below, sloped] {$\alpha_{2}$};
\draw (a2) -- (a3) node [red, midway, below, sloped] {$\alpha_{1}$};

\draw (a4) -- (a5) node [red, midway, below, sloped] {$\alpha_{1}$};

\draw (a2) -- (a4) node [red, midway, below, sloped] {$\alpha_{2}$};
\draw (a3) -- (a5) node [red, midway, below, sloped] {$\alpha_{2}$};

\draw (a5) -- (a6) node [red, midway, below, sloped] {$\alpha_{1}$};


\node at (6, 0) [root] (b1) {$\xi_{1}$};
\node at (7, 1) [root] (b2) {$\xi_{2}$};
\node at (8, 2) [root] (b3) {$\xi_{5}$};

\node at (8, 0) [root] (b4) {$\xi_{3}$};
\node at (9, 1) [root] (b5) {$\xi_{4}$};

\node at (10, 0) [root] (b6) {$\xi_{6}$};

\draw (b1) -- (b2) node [red, midway, below, sloped] {$\alpha_{2}$};
\draw (b2) -- (b3) node [red, midway, below, sloped] {$\alpha_{1}$};

\draw (b4) -- (b5) node [red, midway, below, sloped] {$\alpha_{1}$};

\draw (b2) -- (b4) node [red, midway, below, sloped] {$\alpha_{2}$};
\draw (b3) -- (b5) node [red, midway, below, sloped] {$\alpha_{2}$};

\draw (b5) -- (b6) node [red, midway, below, sloped] {$\alpha_{1}$};

\end{tikzpicture}
\end{center}

\caption{The counterexample illustrated: to the left the poset of radical roots, to the right its numbering according to the chosen total order.}
\label{fig:counter}

\end{figure}

As mentioned during the proof, for this example the property $\xi < \alpha$ does not hold for all radical roots $\xi$ and non-radical roots $\alpha$. Recall that we assumed this property in the setting of \autoref{sec:convex-ord}. For this reason we make the following conjecture.

\begin{conjecture}
\label{con:conjecture}
\autoref{prob:equality} has a positive answer if $w_0$, the longest word of the Weyl group of $\mathfrak{g}$, is factorized as $w_0 = w^\prime w_{0, \mathfrak{l}}$ or similarly as $w_0 = w_{0, \mathfrak{l}} w^{\prime \prime}$, for some $w^\prime$ and $w^{\prime \prime}$. Here $w_{0, \mathfrak{l}}$ is the longest word of the Weyl group of (the semisimple part of) the Levi factor $\mathfrak{l}$.
\end{conjecture}

\vspace{3mm}

{\footnotesize
\emph{Acknowledgements}.
I am supported by the European Research Council under the European Union's Seventh Framework Programme (FP/2007-2013) / ERC Grant Agreement no. 307663 (P.I.: S. Neshveyev).
}

\bigskip

\end{document}